\documentclass[]{article}
\usepackage{authblk}
\usepackage[letterpaper, left=3cm, right=3cm, height=22cm]{geometry}
\usepackage[english]{babel}
\usepackage[utf8]{inputenc}
\usepackage[charter]{mathdesign}
\usepackage[mathscr]{euscript}
\usepackage{amsmath,amsthm}
\usepackage{tikz-cd}
\usepackage[all]{xy}
\usepackage{todonotes}
\usepackage{hyperref}
\usepackage{bussproofs}
\newtheorem{teorema}{Theorem}[section]
\newtheorem{cor}[teorema]{Corollary}
\newtheorem{lem}[teorema]{Lemma}
\newtheorem{prop}[teorema]{Proposition}

\theoremstyle{definition}
\newtheorem{defin}[teorema]{Definition}
\newtheorem{re}[teorema]{Remark}

\newtheorem{Exs}[teorema]{Examples}

\newcommand{\ff}{\mathfrak{f}}
\newcommand{\fg}{\mathfrak{g}}
\newcommand{\fx}{\mathfrak{x}}
\DeclareMathOperator{\downc}{\downarrow\!}

\newcommand{\kmodto}{\mathrel{\mathmakebox[\widthof{$\xrightarrow{\rule{1.45ex}{0ex}}$}]{\xrightharpoonup{\rule{1.45ex}{0ex}}\hspace*{-2.8ex}{\circ}\hspace*{1ex}}}}

\newcommand{\Multi}[1]{(L,{#1})\mbox{-} \mathtt{Cat}}
\newcommand{\Pro}{V \mbox{-} \mathtt{Pro}}
\newcommand{\Pow}[1]{ \mathtt{P}_{{#1} }}
\newcommand{\Vcat}{V \mbox{-} \mathtt{Cat}}
\newcommand{\Dist}[1]{{#1} \mbox{-}\mathtt{Dist}}
\newcommand{\Mat}[1]{ {#1} \mbox{-}\mathtt{Rel}}

\newcommand{\Alg}[1]{\mathtt{Set}^{P_{#1}}}
\newcommand{\Tcolim}[1]{ {#1}\mbox{-} \mathtt{colim}}
\newcommand{\limit}{\mathtt{lim}}
\newcommand{\colim}{\mathtt{colim}}
\newcommand{\limp}{\lhd}
\newcommand{\Yo}[1]{\mathbf{y}_{#1}}
\newcommand{\CoYo}[1]{\mathbf{\lambda}_{#1}}
\newcommand{\rimp}{\rhd}
\newcommand{\mult}{\ast}
\newcommand{\op}{\mathtt{op}}
\newcommand{\Sets}{\mathtt{Set}}
\newcommand{\Id}{\mathtt{Id}}
\newcommand\setItemnumber[1]{\setcounter{enumi}{\numexpr#1-1\relax}}
\usepackage{mathtools}

\makeatletter
\def\slashedarrowfill@#1#2#3#4#5{%
	$\m@th\thickmuskip0mu\medmuskip\thickmuskip\thinmuskip\thickmuskip
	\relax#5#1\mkern-7mu%
	\cleaders\hbox{$#5\mkern-2mu#2\mkern-2mu$}\hfill
	\mathclap{#3}\mathclap{#2}%
	\cleaders\hbox{$#5\mkern-2mu#2\mkern-2mu$}\hfill
	\mkern-7mu#4$%
}
\def\rightslashedarrowfill@{%
	\slashedarrowfill@\relbar\relbar\mapstochar\rightarrow}
\newcommand\xslashedrightarrow[2][]{%
	\ext@arrow 0055{\rightslashedarrowfill@}{#1}{#2}}
\makeatother

\newcommand\restr[2]{{
		\left.\kern-\nulldelimiterspace 
		#1 
		\vphantom{\big|} 
		\right|_{#2} 
}}

\newcommand{\scomment}[1]{\ \ifhmode\todo{#1}\else\vadjust{\todo{#1}}\fi}

\tikzset{%
	symbol/.style={%
		draw=none,
		every to/.append style={%
			edge node={node [sloped, allow upside down, auto=false]{$#1$}}}
	}
}

\title{Injective Hulls of Quantale-Enriched Multicategories}
\author{Eros Martinelli}
\providecommand{\keywords}[1]{\textbf{\textit{Keywords:}} #1}
\affil{Center for Research and Development in Mathematics and
Applications,\\
Department of Mathematics,\\
University of Aveiro,\\
Portugal.\\
\textit{E-mail adress}: \texttt{eros.martinelli@ua.pt}}
\date{}

\begin{document}

\maketitle
\begin{abstract}
	In this communication we generalize some recent results of Rump to categories
	enriched in a commutative quantale $V$. Using these results, we show that every
	quantale-enriched multicategory admits an injective hull. Finally, we expose a
	connection between the Isbell adjunction and the construction of injective hulls
	for topological spaces made by Banaschewski in $1973$.
\end{abstract}
\keywords{Quantales, Quantale-Enriched Categories, Injective Objects, Injective Hulls, Quantale-Enriched Multicategories.}
\medskip

\section{Introduction}
In their article \cite{InjHulls}, Lambek et al. studied injective objects and injective hulls in the category of ordered monoids with sub-multiplicative monotone maps. They proved that quantales are the injective objects in this category and that every ordered monoid admits an injective hull which is obtained as the set of fixed points of a quantic nucleus on the quantale of down-closed subsets of the given ordered monoid.\par\medskip
Rump, in his article \cite{2016THECO}, studied further the injectivity problem by considering quantum B-algebras: ordered sets equipped with two binary operations that mimic the residuals in a quantale. He proved that quantales are the injective objects in the category whose objects are quantum B-algebras and whose morphisms are oplax homomorphisms between them. Moreover, he proved that every quantum B-algebra admits an injective hull.\par\medskip
In this paper we show how these two results can be put under a common roof and generalized to the realm of enriched categories; $V$-multicategories, for a quantale $V$, provide this common roof. In this way, both ordered monoids with sub-multiplicative monotone maps and quantum B-algebras with oplax homomorphisms become full subcategories of $\Multi{V}$: the category whose objects are $V$-multicategories and whose morphisms are $V$-multifunctors, which are able to capture both the sub-multiplicativity and the oplaxness.
We also notice that $V$-multicategories are particular examples of $(T,V)$-categories (see \cite{hofmann_seal_tholen_2014}) where the monad $T$ is specialized to the list monad $L$. For these categories, from now on called $(L,V)$-categories, Clementino, Hofmann and Tholen showed how it is possible to develop many constructions that come from enriched category theory in the more general context of $(T,V)$-categories (see \cite{HOFMANN2011283, CH09a, Hof14}). In particular, in \cite{HOFMANN2011283}, Hofmann showed, by generalizing some results about topological spaces (see \cite{Esc97,Esc98}), how algebras for a Kock-Zöberlein monad, which generalizes the presheaf monad, characterize injective objects in the category of all $(T,V)$-categories.\par\medskip
By using this characterization of injective objects as algebras for a monad, by pure categorical arguments, we generalize Theorem $4.1$ of \cite{InjHulls} and prove that quantales (in fact, their enriched counterparts) are also the injective objects in $\Multi{V}$. Since quantales can be seen both as ordered monoids and as quantum B-algebras, this allows us to recover the characterization of injectives contained in \cite{InjHulls} and in \cite{2016THECO}.\par\medskip
Regarding the problem of the construction of injective hulls, using ideas from \cite{2016THECO}, we prove that, for a full subcategory $C$ of $\Multi{V}$, every object of $C$ admits an injective hull if and only if the collection of objects of $C$ contains all quantales.\par\medskip

We would like to stress the fact that the results of this paper are not a mere exercise in \textit{fuzzyfication}, that is to say, a rephrase of the ones contained in \cite{2016THECO} for a general $V$. In order to be able to generalize them, one needs a more general theory of colimits, which is provided by the notion of $(L,V)$-colimits which we are going to introduce. This provides a nice and elegant categorical way to shed light on some universal constructions that in the ordered case are somehow "hidden".\par\medskip

The structure of the paper is as follows:
\begin{itemize}
	\item In the first section we introduce some background material on $V$-categories and some "classical" constructions: distributors, (co)limits, and adjunctions.
	\item In the second section we introduce $(L,V)$-categories. We show how many constructions that come from enriched category theory (distributors, colimits, colimit completions etc.) can be developed in the more general context of $(L,V)$-categories.
	\item In the third section we define $V$-quantales, the enriched counterparts of quantales, state some of their properties and introduce lax monoidal monads.
	\item The fourth section is, as the title suggests, an \textit{intermezzo}. We recall how injectives and injective hulls are built in the category of ordered sets and monotone maps. Then we show how the same ideas apply to the enriched case, for a general commutative quantale $V$. The aim of the section is to introduce some ideas that will be applied later in the more general context of $(L,V)$-categories.
	\item The fifth section forms the core of the paper. We introduce promonoidal categories and we explain their relation with $(L,V)$-categories. We introduce enriched quantum B-algebras as representable promonoidal categories. We prove that $V$-quantales are precisely the injectives in $\Multi{V}$. Using some ideas from the previous sections we prove that every enriched quantum B-algebra admits an injective hull.
	\item In the sixth section we characterize those subcategories of $\Multi{V}$ for which injective hulls exist. Using ideas from the previous section we prove that, for a full subcategory $C$ of $\Multi{V}$, every object of $C$ admits an injective hull if and only if the collection of objects of $C$ contains all $V$-quantales.
	\item The final section sketches a connection between the Isbell adjunction in the ordered case and the construction, exposed in \cite{Ban73}, of injective hulls of topological spaces.
\end{itemize}

\section{Preliminaries on Quantale-Enriched Categories}
Enriched categories are a generalization of ordinary categories in which \textit{hom-sets} take values in a \textit{cosmos}, a symmetric-closed monoidal category $V$. The standard reference about them is \cite{ECT}.\par\medskip

In this document we are going to consider only the case in which $V$ is a \textit{commutative quantale}, a commutative monoid in the monoidal category of sup-lattices and suprema preserving monotone maps. Although this might be seen as not so useful, Lawvere, in his seminal paper \cite{MSCC}, showed how taking enrichment in a commutative quantale is not only useful as a toy model in which some general constructions become easier to understand (due to the absence of coherence conditions), but how quantale-enriched categories are worthy of being studied on their own, since they are able to capture important mathematical structures like metric spaces.\par\medskip

In the following section we recall/introduce some basic notions of $V$-categories. Our point of view is slightly different from the more "standard" one contained in \cite{ECT}, it is more "relational": following \cite{BETTI1983109, CLEMENTINO200313}, we introduce the \textit{quantaloid} of $V$-relations and we define $V$-categories starting from there. This might be seen as an overkill, but it will be clear in the section related to $(L,V)$-categories how this approach allows us to smoothly introduce some concepts (as distributors, presheaves and colimits) also in the $(L,V)$-case.
\subsection{V-Categories and V-Functors}
\begin{defin}
	A  commutative quantale $(V, \otimes, k)$ is a complete lattice endowed with a (commutative) multiplication $\otimes : V \times V \rightarrow V$ that preserves suprema in each variable and for which $k\in V$ is the neutral element. If $k \neq \perp,$ we call $V$ non-trivial.
\end{defin}
\begin{re}
	In this paper we assume---unless explicitly stated---quantales in which we take the enrichment to be non-trivial.
\end{re}
\begin{defin}
	Let $(V, \otimes_V, k_V)$ and $(Q, \otimes_Q, k_Q)$ be commutative quantales. A \textit{lax morphism of quantales} is a monotone map $f : V \rightarrow Q$ such that, for all $v,w \in V$,
	$$ f (v) \otimes_Q f(w) \leq f(v \otimes_V w), \ \ k_Q \leq f(k_V).$$
	If $f$ preserves suprema and if, for all $v,w \in V$,
	$$ f (v) \otimes_Q f(w) = f(v \otimes_V w), \ \ k_Q  = f(k_V),$$
	then $f$ is called \textit{strong}.
\end{defin}
\begin{re}
	By the adjoint functor theorem applied to ordered sets, it follows that $- \otimes =$ admits a right adjoint (in each variable) denoted by $[-,=]$ and called "internal hom".
\end{re}
\begin{Exs}
	\begin{enumerate}
		\item The two-element boolean algebra $\mathbf{2} = \{0,1\}$ with $\wedge$ as multiplication and $\Rightarrow$ as internal hom is a quantale.
		\item More generally, every frame becomes a quantale with the multiplication given by $\wedge$. In particular $k = \top.$
		\item $[0, \infty]^{\op}$ (with the opposite of the natural order) with $+$ as multiplication is a quantale. The internal hom is given by "truncated minus" defined as $ [u,v] = v\ominus u = \mathtt{max}(v-u,0)$.
		\item Consider the set
		$$\Delta = \{\psi :[0, \infty] \rightarrow [0,1] \mbox{ $\mid$ } \mbox{for all } \alpha \in [0, \infty] \mbox{ : } \psi(\alpha) = \bigvee_{\beta < \alpha } \psi(\beta)\}$$
		of distribution functions; with the pointwise order it becomes a complete ordered set. For all $\psi, \phi \in \Delta$ and $\alpha \in [0, \infty]$, define the following multiplication:
		$$ \psi \otimes \phi (\alpha) = \bigvee_{\beta + \gamma < \alpha } \psi(\beta )\mult \phi(\gamma),$$
		where $\mult$ is the ordinary multiplication on $[0,1]$.
		It is shown in \cite{De} that $(\Delta, \otimes, k)$ is a quantale, where $k(0) =0$ and, for all $\alpha > 0$, $k(\alpha) = 1$.
\end{enumerate}
\end{Exs}
\begin{Exs}Let $(V, \otimes, k)$ be a commutative quantale.
	\label{QMorph}
	\begin{enumerate}
			\item \label{QuantMorph3} The function $f : \mathbf{2} \rightarrow V$, where
			$$
			f(v) =
			\begin{cases}
			\perp &\mbox{if } v = 0, \\
			k &\mbox{if } v = 1,
			\end{cases}
			$$
			defines a strong morphism of quantales.
			\item \label{QuantMorph2} The function $h : V \rightarrow \mathbf{2}$, where
			$$
			h(v) =
			\begin{cases}
			1  &\mbox{if } v  \neq \perp, \\
			0 &\mbox{if } v = \perp,
			\end{cases}
			$$
			in general does not define a morphism of quantales. The function $h$ defines a morphism of quantales iff $v \otimes u = \perp$ implies that $v = \perp$ or $u = \perp $, for all $u,v \in V$. We notice that the quantale $[0, \infty]^{\op}$ satisfies this condition.
			\item \label{QuantMorph} The function $g: V \rightarrow \mathbf{2}$, where
			$$
			g(v) =
			\begin{cases}
		 		1 &\mbox{if } k \leq v, \\
			k &\mbox{otherwise, }
			\end{cases}
			$$
			is a lax morphism of quantales.
	\end{enumerate}
\end{Exs}
As we stated in the introduction of this section, we are going to present $V$-categories from a more "relational" point of view. The first step is to define the so-called \textit{quantaloid of $V$-relations}, which is the enriched generalization of the category $\mathtt{Rel}$ of (ordinary) binary relations. For an account on quantaloids we refer to \cite{STUBBE201495} for a brief overview and to \cite{DF} for a more in depth description.\par\medskip

The quantaloid $\Mat{V}$ is the order-enriched category whose objects are sets, and an arrow $r : X \xslashedrightarrow{} Y$ is given by a function
$$ r : X \times Y \rightarrow V.$$
The composition of $r : X \xslashedrightarrow{} Y $, $s : Y \xslashedrightarrow{} Z$ is given by "matrix multiplication" and is defined pointwise as
$$ s \cdot r (x,z) = \bigvee_{y \in Y} r(x,y) \otimes s(y,z).$$
The identity arrow $ \Id : X \xslashedrightarrow{} X$ is defined as
$$
\Id(x_1,x_2) =
\begin{cases}
k &\mbox{if } x_1 = x_2, \\
\perp &\mbox{if } x_1 \neq x_2.
\end{cases}
$$
The complete order on $\Mat{V}(X,Y)$ is the one induced (pointwise) by $V$, i.e.,
\begin{equation}
\label{Ord}
 r \leq r' \mbox{ in } \Mat{V}(X,Y) \mbox{ iff } r(x,y) \leq r'(x,y) \mbox{ in } V \mbox{ for all } x,y \in X,Y.
\end{equation}
\begin{re}
	Notice that $\Mat{V}(X,Y)$ is complete because $V$ is so. Since the multiplication of $V$ preserves suprema in both variables and because suprema commute with suprema, one has
	$$(\bigvee_i r_i) \cdot (\bigvee_j s_j) = \bigvee_{i,j} r_i \cdot s_j.$$
	This proves that $\Mat{V}$ is a quantaloid.
 \end{re}
\begin{re}
	Notice that for $V = \mathbf{2}$, $\Mat{\mathbf{2}}$ is the quantaloid of relations, and the "matrix multiplication" defined previously becomes the "classical" relational composition.
\end{re}
\begin{re}
	Note that \eqref{Ord} is equivalent to
	$$ k \leq \bigwedge_{x,y \in X,Y} [r(x,y), r'(x,y)];$$
	indeed, consider $x,y \in X,Y$, then we have
	$$ r(x,y) \leq r'(x,y) \mbox{ iff }  k \otimes r(x,y) \leq r'(x,y) \mbox{ iff } k \leq [r(x,y), r'(x,y)].$$
\end{re}
\begin{re}\label{ChangeOfBase}
Let $ f: V \rightarrow Q$ be a lax morphism between two quantales $(V, \otimes_V, k_V)$ and $(Q,\otimes_Q,k_Q)$. Then $f$ induces a lax functor
$$ F : \Mat{V} \rightarrow \Mat{Q},$$
which is the identity on objects and which sends a $V$-relation $r : X \xslashedrightarrow{} Y$ to the $Q$-relation
$$F(r) :  X \times Y \xrightarrow{r} V \xrightarrow{f} Q.$$
It is easy to prove that $F$ is lax; indeed, one has (for suitable $V$-relations $r,s$)
\begin{alignat*}{2}
 F(r) \cdot F(s)(x,z) &= \bigvee_{y \in Y} F(r)(x,y) \otimes_Q F(s)(y,z)  \\
       &\leq \bigvee_{y \in Y} f( r(x,y) \otimes_V s(y,z) ) \\
			 & \leq F(r \cdot s)(x,z).
\end{alignat*}
Notice that if $f$ is a strong morphism between quantales then $F$ is strong too.
\end{re}
\begin{re}
	Notice that every function $f : X \rightarrow Y$ can be seen as a $V$-relation as follows:
	$$
	f(x,y) =
	\begin{cases}
	k &\mbox{if } f(x) =y, \\
	\perp &\mbox{if } f(x)\neq y.
	\end{cases}
	$$
	The identity in $\Mat{V}(X,X)$ is an example of this construction.
\end{re}

We have also an involution $(-)^{\circ} : \Mat{V}^{\op} \rightarrow \Mat{V}$ defined as $r^{\circ}(y,x) = r(x,y)$ which satisfies
$$ (1_X)^{\circ} = 1_X, \ \ (s \cdot r )^{\circ} = r^{\circ} \cdot s^{\circ}, \ \ (r^{\circ})^{\circ} = r.$$
\begin{defin}
	A $V$-category is a pair $(X,a)$, where $X$ is a set and $ a : X \xslashedrightarrow{} X$ is a $V$-relation that satisfies
	\begin{itemize}
		\item $\Id \leq a;$
		\item  $a \cdot a \leq a$.
	\end{itemize}
\end{defin}
\begin{re}
	In this paper, when the $V$-structure is clear from the context, we will denote a $V$-category $(X,a)$ simply as $X$.
\end{re}
\begin{defin}
	Given two $V$-categories $(X,a), (Y,b)$, a $V$-\textit{functor} $f : (X,a) \rightarrow (Y,b)$ is a function between the underlying sets such that
	$$a \leq f^{\circ} \cdot b \cdot f$$
	which, in pointwise terms, means that, for all $x,y \in X,$
	$$ a(x,y) \leq b(f(x), f(y)).$$
	If the equality holds, we call $f$ \textit{fully faithful}.
\end{defin}
\begin{Exs}

	\begin{enumerate}
		\item For $V = \mathbf{2}$, as we already stated before, a $\mathbf{2}$-category is an ordered set and a $\mathbf{2}$-functor is a monotone map. Notice that the order relation of a $\mathbf{2}$-category $(X, \leq_X)$ does not need to be antisymmetric.
		\item Categories enriched in the quantale $[0, \infty]^{\op}$, as first recognized by Lawvere in \cite{MSCC}, are generalized metric spaces and $[0, \infty]^{\op}$-functors between them are non-expansive maps.
		\item Categories enriched in $\Delta$ are called probabilistic metric spaces, as first recognized in \cite{De}.
		\item The quantale $V$ defines a $V$-category with the $V$-structure given by its internal hom $[-, =]$.
		\item By using the involution $(-)^{\circ} : \Mat{V}^{\op} \rightarrow \Mat{V}$, for every $V$-category $(X,a)$, one can define its opposite category $X^{\op} = (X, a^{\circ}).$
		\item Given two $V$-categories $(X,a), (Y,b)$, one can define the $V$-category formed by all $V$-functors $f : (X,a) \rightarrow (Y,b)$, denoted by $([X,Y],[X,Y](-,=) )$, by defining the following $V$-structure:
		$$[X,Y](f,g) = \bigwedge_{x\in X} b(f(x),g(x)).$$
		In particular we have two very important $V$-categories:
		$$\mathbb{D}(X)= [X^{\op},V], \mbox{ the category of presheaves},$$
		$$\mathbb{U}(X) = [X,V]^{\op}, \mbox{ the category of co-presheaves.}$$
		Notice that they are generalizations (for a general $V$) of the classical down(up)-closed subsets construction which corresponds to the case in which $V= \mathbf{2}$.
		\item Given a $V$-category $(X,a)$, there are two $V$-functors, called the Yoneda embedding and the co-Yoneda embedding:
		$$ \Yo{X} : X \rightarrow \mathbb{D}(X), \ \ x \mapsto a(-,x),$$
		$$ \CoYo{X} : X \rightarrow \mathbb{U}(X), \ \ x \mapsto a(x,=).$$
		Moreover, one can prove that
		$$ \mathbb{U}(X)[\CoYo{X}(x),g] = g(x), \ \ \mathbb{D}(X)[\Yo{X}(x),g] = g(x).$$
		The last two results are known as the co-Yoneda lemma and Yoneda lemma, respectively. Notice that, for a general $X$, $\Yo{X}$ and $\CoYo{X}$ are not injective functions. They are injective iff $X$ is separated (see \cite[Proposition~1.5]{HT10}).
		\item Given two $V$-categories $(X,a), (Y,b)$, we can define their tensor product
		$$ X \boxtimes Y = (X\times Y, a \otimes b).$$
		In particular, one has: $X \boxtimes K \simeq X$ where $K$ denotes the one-point $V$-category $(1,k)$.\par\medskip
		In the case in which $V = \mathbf{2}$, the ordered structure on $X \boxtimes Y$ is the product order. This means that $(x_1,y_1) \leq_{X \boxtimes Y} (x_2, y_2)$ if and only if $x_1 \leq_X x_2$ and $y_1 \leq_Y y_2$.\par\medskip
		In the case in which $V= [0, \infty]^{\op}$, the metric structure on $X \boxtimes Y$ is the taxicab metric, which is defined as:
		$$ d_{X \boxtimes Y }((x_1,y_1),(x_2, y_2) ) = d_X(x_1, x_2) + d_Y(y_1,y_2).$$
	\end{enumerate}

\end{Exs}

In this way we can define $\Vcat$ as the category whose objects are $V$-categories and whose arrows are $V$-functors; moreover, $\Vcat$ becomes an order-enriched category, if we define, for two $V$-functors $f,g : (X,a) \rightarrow (Y,b)$,
$$ f \leq g \mbox{  iff }  k \leq \bigwedge_{x\in X} b(f(x),g(x)).$$
With the tensor product previously defined, $\Vcat$ becomes a closed monoidal category, since one can show that, for three $V$-categories $(X,a), (Y,b),(Z,c)$, one has
$$\Vcat(X \boxtimes Y,Z) \simeq \Vcat(X, [Y,Z]) \simeq \Vcat(Y, [X,Z]).$$
We have that $\Vcat$, with the product defined before, becomes a monoidal category. This allows us to define monoids with respect to such product, which we call monoidal $V$-categories.
\begin{defin}
A monoidal $V$-category $(X,a, \mult, u_X)$ is a $V$-category $(X,a)$ equipped with two $V$-functors: $\mult : X \boxtimes X \rightarrow X$ and $ u_X : K \rightarrow X$, such that $(X,a, \mult, u_X)$ is a monoid (with respect to the monoidal structure $(\boxtimes, K)$).
\end{defin}
\begin{re}
Notice that $K = (1,k)$ is a separator in $\Vcat$. This means that, for all pairs of parallel $V$-functors $f,g : X \rightarrow Y$ in $\Vcat$, if $f z = g z$ for every $V$-functor $z : K \rightarrow X$, then $f = g$.
\end{re}
\begin{re}
	Notice that, when $V=\mathbf{2}$, a monoidal ordered set is just an ordered monoid. That is to say a monoid endowed with an order relation which is compatible with the monoid structure.\\
	When $V = [0, \infty]^{\op}$, a monoid in $ [0, \infty]^{\op} \mbox{-} \mathtt{Cat}$ is a metric space endowed with a (compatible with respect to the metric) monoid structure on its underlying set. Examples of monoidal metric spaces are the underlying additive groups of normed vector spaces.
\end{re}
\begin{re}\label{ChangeOfBase2}
	Let $ f: V \rightarrow Q$ be a lax morphism between two quantales $(V, \otimes_V, k_V), (Q,\otimes_Q,k_Q)$. In Remark \ref{ChangeOfBase} we showed that $f$ induces a lax functor $F : \Mat{V} \rightarrow \Mat{Q}.$ It is possible to show that $F$ allows us to define a $2$-functor
	$$\overline{F} : \Vcat \rightarrow Q \mbox{-} \mathtt{Cat},$$
	which sends a $V$-category $(X,a)$ to the $Q$-category $(X, F(a))$ and a $V$-functor $f : (X,a) \rightarrow (Y,b)$ to the $Q$-functor $\overline{F}(f) : (X, F(a)) \rightarrow (Y,F(b))$ that has the same underlying function as $f$.\\
	In particular, if we consider the lax morphism of quantales $g : V \rightarrow \mathbf{2}$, introduced in Example \ref{QuantMorph} of Examples \ref{QMorph}, we get a $2$-functor
	$$\Vcat \rightarrow \mathtt{Ord}$$
	which sends a $V$-category $(X,a)$ to the ordered set $(X,\leq_a)$, where, for $x,y \in X$,
	$$ x \leq_a y \mbox{ iff } k \leq a(x,y).$$
	Moreover, in the case in which $V = [0, \infty]^{\op} $, we have that
	$$ [0, \infty]^{\op} \mbox{-} \mathtt{Cat} \rightarrow \mathtt{Ord}$$
	preserves the monoidal structure. This means that, for all $[0, \infty]^{\op}$-categories $(X, d_X)$ and $(Y,d_Y)$, the underlying ordered set of $(X\boxtimes Y, d_{X\boxtimes Y} )$ is $(X\boxtimes Y, \leq_{X\boxtimes Y} )$. Indeed, one has
	\begin{alignat*}{2}
	d_X(x_1, x_2) + d_Y(y_1,y_2) \leq 0 & \mbox{ iff } d_X(x_1, x_2) = 0 \mbox{ and } d_Y(y_1,y_2)= 0 \\
	& \mbox{ iff } x_1 \leq_{d_X} x_2 \mbox{ and } y_1 \leq_{d_Y} y_2 \\
	& \mbox{ iff } (x_1,y_1) \leq_{X\boxtimes Y} (x_2, y_2).
	\end{alignat*}
	In the same way, if we consider a quantale $V$ such that $h: V \rightarrow \mathbf{2}$---introduced in Example \ref{QuantMorph2} of Examples \ref{QMorph}---is a strong morphism of quantales, we get another $2$-functor
	$$\Vcat \rightarrow \mathtt{Ord}$$
	which sends a $V$-category $(X,a)$ to the ordered set $(X,\leq_{h(a)})$, where, for $x,y \in X$,
	$$ x \leq_{h(a)} y \mbox{ iff }  a(x,y) \neq \perp.$$
	Moreover, in the case in which $V = [0, \infty]^{\op} $, as before, we have that
	$$ [0, \infty]^{\op} \mbox{-} \mathtt{Cat} \rightarrow \mathtt{Ord}$$
	preserves the monoidal structure. Indeed one has
	\begin{alignat*}{2}
	d_X(x_1, x_2) + d_Y(y_1,y_2) \neq  \perp & \mbox{ iff } d_X(x_1, x_2) \neq  \perp \mbox{ and } d_Y(y_1,y_2) \neq  \perp\\
	& \mbox{ iff } x_1 \leq_{h(d_Y)} x_2 \mbox{ and } y_1 \leq_{h(d_X)} y_2 \\
	& \mbox{ iff } (x_1,y_1) \leq_{h(d_{X \boxtimes Y})} (x_2, y_2).
	\end{alignat*}
\end{re}
\subsection{Distributors and Weighted (co)Limits}
Bénabou introduced distributors in \cite{LD} and since then they played an important role in category theory. They can be seen as generalizations of ideal relations from order theory, that is to say, subsets of the cartesian product of two ordered sets $X, Y$ which are downward closed in $X$ and upward closed in $Y$.\par\medskip

Weighted (co)limits encompass the classical notion of (co)limits coming from "ordinary" category theory, by admitting a "weight" given by a distributor. In the case of ordered sets, due to the fact that every distributor is a (particular) subset, the weight "disappears" and becomes the "set of conditions" over which suprema/infima are taken.
\begin{defin}
Given two $V$-categories $(X,a), (Y,b)$, a $V$-\textit{distributor} (or simply a distributor) $j  : (X,a) \xslashedrightarrow{} (Y,b)$ is a $V$-relation between the underlying sets such that:
\begin{itemize}
	\item $j \cdot a \leq j;$
	\item $b \cdot j \leq j.$
\end{itemize}
\end{defin}
Since the composite of two distributors is again a distributor, we can define the quantaloid $\Dist{V}$ in the same way as we defined $ \Mat{V}.$ In $\Dist{V}(X,X)$ the $V$-structure $a$ plays the role of the identity, since for every distributor $j : X \xslashedrightarrow{} Y$, one has
$$ b \cdot j = j \cdot a = j.$$
\begin{re}\label{ChangeOfBase3}
	Let $ f: V \rightarrow Q$ be a lax morphism between two quantales $(V, \otimes_V, k_V)$ and $(Q,\otimes_Q,k_Q)$. In Remark \ref{ChangeOfBase} we showed that $f$ induces a lax functor
	$$ F : \Mat{V} \rightarrow \Mat{Q}.$$
	It is easy to show that $F$ sends $V$-distributors to $Q$-distributors and thus that it induces a lax functor
	$$ \Dist{V} \rightarrow \Dist{Q}.$$
	Indeed, let $j : X \xslashedrightarrow{} Y$ be a $V$-distributor. Then, for all $x \in X$, $y\in Y$,
	\begin{alignat*}{2}
	F(j) \cdot F(a)(x,y) &= \bigvee_{z \in X} F(a)(x,z) \otimes_Q F(j)(z,y) \\
	& \leq \bigvee_{z \in X} F(a(x,z) \otimes_V j(z,y)) \\
	& \leq F( \bigvee_{z \in X} a(x,z) \otimes_V j(z,y)) \\
	& = F (j \cdot a)(x,y) \\
	& \leq F(j)(x,y).
	\end{alignat*}
	In the same way one can prove that $F(b) \cdot F(j) \leq F(j).$\\
	Notice that, if we have a strong morphism of quantales, then $ \Dist{V} \rightarrow \Dist{Q}$ is strong too.
\end{re}
\begin{re}
	By juggling with the definition of distributor, one can show that distributors between two $V$-categories $(X,a), (Y,b)$ are in bijective correspondence with $V$-functors between $X^{\op} \boxtimes Y$ and $V$. It is easy to prove that this correspondence is functorial and that it gives an equivalence of ordered sets
	$$\Dist{V}(X,Y) \simeq \Vcat(X^{\op} \boxtimes Y,V) \simeq \Vcat(Y,\mathbb{D}(X)).$$
	In particular, to every $V$-distributor $j : X \xslashedrightarrow{} Y$ we can associate its \textit{mate}
	$$ \ulcorner j \urcorner : Y \rightarrow \mathbb{D}(X), \ \ y \mapsto j(-,y).$$
\end{re}
Given a $V$-functor $f : (X,a) \rightarrow (Y,b)$, we can define two arrows in $\Mat{V}$:
\begin{itemize}
	\item $f_{*} : X \xslashedrightarrow{} Y,  \ \ f_{*}(x,y) = b(f(x),y);$
	\item  $f^{*} : Y \xslashedrightarrow{} X, \ \ f^{*}(y,x) = b(y,f(x)).$
\end{itemize}
One has:
\begin{lem}
 The $V$-relations $f_{*}$ and $f^{*}$ are both distributors, moreover, $f_{*} \dashv f^{*}$ in $\Dist{V}.$
\end{lem}
In this way we have two $2$-functors
$$ (-)_{*} : \Vcat^{\mathtt{co}} \rightarrow \Dist{V}, \ \  (- )^{*}:  \Vcat^{\op} \rightarrow \Dist{V}.$$
\begin{teorema}
$(-)_{*} : \Vcat^{\mathtt{co}} \rightarrow \Dist{V}$ defines a proarrow equipment \cite{CTGDC_1982__23_3_279_0, CTGDC_1985__26_2_135_0} on $\Vcat$, that is to say:
\begin{itemize}\label{Pro}
	\item $(-)_{*}$ is locally fully faithful;
	\item for all $V$-functors $f : (X,a) \rightarrow (Y,b)$, there exists a right adjoint to $f_{*}$ in $\Dist{V}.$
\end{itemize}
\end{teorema}
\begin{proof}
	For the first point we have to show that
	$$ f \leq g \mbox{ in } \Vcat \mbox{ iff }  g_{*} \leq f_{*} \mbox{ in } \Dist{V}.$$
	$\Rightarrow).$ Fix an element $x \in X$. From $k \leq \bigwedge_{x'\in X} b(f(x'),g(x'))$ we get $k \leq b(f(x), g(x))$. By applying composition we get
	$$ b(g(x), y ) \leq b(f(x), g(x)) \otimes b(g(x), y) \leq b(f(x),y).$$

$\Leftarrow).$ For every $x \in X$, we get
	$$ k \leq b(g(x),g(x) ) \leq b(f(x), g(x)),$$
	hence
	$$k \leq \bigwedge_{x'\in X} b(f(x'),g(x')).$$
	The second point follows from the previous lemma. \\
\end{proof}
As a corollary we get:
\begin{cor}\label{DistFun}
	A function $f: (X,a) \rightarrow (Y,b)$ between $V$-categories is a $V$-functor iff $f_{*}: X \xslashedrightarrow{} Y$ is a $V$-distributor; equivalently iff $f^{*} : Y \xslashedrightarrow{} X$ is a $V$-distributor.
\end{cor}
The theorem above allows us to "mirror" all the good features of $\Dist{V}$ in $\Vcat$, giving us a "higher" perspective on notions like:
weighted (co)limits, (co)ends, Kan extensions, and adjoint $V$-functors.\\
The key fact is that $\Dist{V}$ is "closed", which means that composition has an adjoint in each variable, more precisely, given two distributors
$$ \alpha : Z \xslashedrightarrow{} X, \ \ \beta : X \xslashedrightarrow{} Y,$$
one has:
\begin{itemize}
	\item $ (-) \cdot \alpha \dashv (-) \rhd \alpha;$
	\item $\beta \cdot (=) \dashv  \beta \lhd (=).$
\end{itemize}
where
\begin{alignat*}{2}
 (-) \cdot \alpha  : \Dist{V}(X,Y) &\rightarrow \Dist{V}(Z,Y) \\
\beta        &\longmapsto \beta \cdot \alpha
\end{alignat*}
\begin{alignat*}{2}
\beta \cdot (=)  : \Dist{V}(Z,X) &\rightarrow \Dist{V}(Z,Y) \\
\alpha     &\longmapsto \beta \cdot \alpha
\end{alignat*}
and, for $\gamma : Z \xslashedrightarrow{} Y,$
$$( \gamma \rhd \alpha ) (x,y) = \bigwedge_{z \in Z} [\alpha(z,x),\gamma(z,y) ],$$
$$(\beta \lhd\gamma  )(z,x) = \bigwedge_{y \in Y} [\beta(x,y) ,\gamma(z,y) ].$$
\begin{defin}
Let $f : (X,a) \rightarrow (Y,b)$ be a $V$-functor, $(Z,c)$ be a $V$-category, and let $j : X \xslashedrightarrow{} Z$ be a distributor. We say that a $V$-functor $g : (Z,c) \rightarrow (Y,b)$ is the \textit{colimit of $f$ weighted by $j$}, and denote it by $\colim(j,f)$, if $g_{*} \simeq f_{*} \rhd  j$, that is to say
$$g_{*}(z,y) = b(g(z), y) = \bigwedge_{x \in X} [j(x,z), f_{*}(x,y)].\ \  (1)$$
Dually, if we have a distributor $ l : Z \xslashedrightarrow{} X$, we say that a $V$-functor $h : (Z,c) \rightarrow (Y,b)$ is the \textit{limit of $f$ weighted by $l$}, and denote it by $\limit(l,f)$, if $h^{*} \simeq l \lhd f^{*}$, that is to say
$$h^{*}(y,z) = b(y,h(z)) = \bigwedge_{x \in X} [l(z,x), f^{*}(y,x)]. \ \ (2)$$
A $V$-category $(Y,b)$ is called \textit{(co)complete} if it has all weighted (co)limits; this means that, for all $V$-functors $f : (X,a) \rightarrow (Y,b)$ and for all distributors ($j : X \xslashedrightarrow{} Z$) $ h : Z \xslashedrightarrow{} X$, the (co)limit of $f$ weighted by $(j)$ $h$ exists.
\end{defin}
\begin{re}
	Using the fact that
	$$\mathbb{D}(X)(j,l) = \bigwedge_{x \in X} [j(x), l(x)],$$
	(and similarly for $\mathbb{U}(X)$) we can re-write $(1)$ and $(2)$ as
	$$b(\colim(j,f)(z), y) \simeq \mathbb{D}(X)(j(-,z),b(f(-), y)), $$
	$$b(y, \limit(l,f)(z)) \simeq \mathbb{U}(X)(b(y,f(=)), l(z, =)).$$
\end{re}
\begin{Exs}
	Let $(X,a)$ be a $V$-category and $x :K \rightarrow X$ be a point (seen as a $V$-functor from the one point $V$-category $K = (1,k)$). Let $u \in V$, we view it as a distributor $ \tilde{u} : K \xslashedrightarrow{} K$ by defining $u(\bullet,\bullet) = u$.
	\begin{enumerate}
		\item The colimit of $x$ weighted by $\tilde{u}$, usually denoted by $x \odot u$, is called \textit{copower} (also: \textit{tensor} in the literature). If we unravel the definition of colimit, we get (for all $y \in Y$)
		$$ a (x \odot u, y) = [u, a(x,y)].$$
		A category in which all copowers exist is called \textit{copowered} (also: \textit{tensored}).
		\item Dually, we can consider also the limit of $x$ weighted by $\tilde{u}$. This is usually denoted by $ x \pitchfork u$ and it is called \textit{power} (also: \textit{cotensor}).\\
		A category in which all powers exist is called \textit{powered} (also: \textit{cotensored}).
\end{enumerate}
Let $f : I \rightarrow X, \ \ i \mapsto x_i$ be a $V$-functor, where the set $I$ is equipped with the $V$-structure
$$
I(i,j) =
\begin{cases}
\perp &\mbox{if } i \neq j, \\
k &\mbox{if } i = j.
\end{cases}
$$
\begin{enumerate}
	 \setItemnumber{3}
	\item Let $ \alpha  : I \xslashedrightarrow{} K$ be the constant distributor $\alpha(i,\bullet) = k$. The colimit of $f$ weighted by $\alpha$ is an element $\tilde{x}$, such that (for all $y \in X$)
	$$ a(\tilde{x}, y) = \bigwedge_{i} a(x_i, y).$$
	Such a colimit is called \textit{conical supremum}, and, since when $V = \mathbf{2}$ it coincides with the supremum of the set $\{x_i \mbox{ | } i \in I\}$, it is usually denoted by $\tilde{x} = \bigvee_{i} x_i.$\\
	\setItemnumber{4}
	\item Dually, let  $\beta : K\xslashedrightarrow{} I$ be the same distributor as above, but seen as an arrow from $K$ to $I$. The limit of $f$ weighted by $\beta$ is an element $\tilde{x}$,  such that (for all $y \in X$)
	$$ a(y, \tilde{x}) = \bigwedge_{i} a(y,x_i).$$
	Such a limit is called \textit{conical infimum}; as before, when $V = \mathbf{2}$ it coincides with the infimum of the the set $\{x_i \mbox{ | } i \in I\}$, and for this reason it is usually denoted by $\tilde{x} = \bigwedge_{i} x_i$.
\end{enumerate}
\end{Exs}
\begin{re}
\label{Deco1}
All the (co)limits we described above play a very important role, as they are basic building blocks: every (co)limit can be written as a suitable composite of them. \\
Let $f : (X,a) \rightarrow (Y,b)$ be a $V$-functor and let $j : X \xslashedrightarrow{} Z$ be a distributor. If $(Y,b)$ is copowered and $\colim(j,f)$ exists, then we have
$$ \colim(j,f)(z) \simeq \bigvee_{x\in X} j(x,z) \odot f(x).$$
Dually, if $(Y,b)$ is powered and $\limit(l,f)$ exists (for a distributor $ l : Z \xslashedrightarrow{} X$), one has
$$ \limit(l,f)(z) \simeq \bigwedge_{x \in X} l(z,x) \pitchfork f(x).$$
\end{re}
\begin{re}
 Just as the set of down-closed subsets and the set of up-closed subsets of an ordered set $X$ have all suprema and infima, for a $V$-category $X$, $\mathbb{U}(X)$ and $\mathbb{D}(X)$ are both complete and (co)complete $V$-categories.
\end{re}
Another important concept we are going to use is the concept of \textit{adjunction} between a pair of $V$-functors
\[
\begin{tikzcd}
(X,a)\ar[r,bend left,"f",""{name=A, below}] & (Y,b). \ar[l,bend left,"g",""{name=B,above}] \ar[from=A, to=B, symbol=\dashv]
\end{tikzcd}
\]
This concept generalizes the classical notion of Galois connection from order theory.
\begin{defin}
Let $f : (X,a) \rightarrow (Y,b)$ and $g : (Y,b) \rightarrow (X,a)$ be two $V$-functors; $f$ is \textit{left adjoint} to $g$ (vice versa: $g$ is \textit{right adjoint} to $f$) and denoted by $f \dashv g$, if $f_{*} \simeq g^{*}$. In pointwise terms, this means that, for all $x,y \in X,Y$,
$$ b(f(x),y) = a(x,g(y)).$$
\end{defin}
\begin{Exs}
	Let $(X,a)$ be a $V$-category.
	\begin{enumerate}
		\item If $x \odot u$ exists for all $u \in V$, then we can rephrase the universal property of the colimit:
		$$ a (x \odot u, y) = [u, a(x,y)],$$
		by saying that $  x \odot =  \dashv a(x,=).$
		\\
		Thus, copowered categories can be equivalently described as those $V$-categories for which, for all $x \in X,$ $a(x,=)$ admits a left adjoint.
		\item If $x \pitchfork u$ exists for all $u \in V$, then we can rephrase the universal property of the limit:
		$$ [u, a(y,x)] = a(y, x \pitchfork u),$$
		by saying that $  x \pitchfork = \dashv a(-,x).$
		\\
		Thus, powered categories can be equivalently described as those $V$-categories for which, for all $x \in X,$ $a(-,x)$ admits a right adjoint.
	\end{enumerate}
\end{Exs}
Let $f: (X,a)\rightarrow (Y,b)$ and $g: (Y,b) \rightarrow (Z,c)$ be two $V$-functors and let $j : X\xslashedrightarrow{} D$ be a distributor. From the definition of weighted colimit it follows that
$\colim(j, gf) \leq g(\colim(j,f))$. \\
Vice versa, if we consider a distributor $l : D \xslashedrightarrow{} X$, from the definition of weighted limit it follows that $g(\limit(l,f)) \leq \limit(h,gf).$\\
If the equality holds, we say that $g$ preserves the (co)limit of $f$ weighted by $(j)$ $h$. If $g$ preserves all (co)limits that exist, we say that $g$ is \textit{(co)continuous}.
\begin{Exs}\label{ColLims}
	Let $(X,a)$ be a $V$-category.
	\begin{enumerate}
		\item The $V$-functor $a(x, =) : X \rightarrow V$ is continuous.
		\item \label{Inv} The $V$-functor $a(-,x) : X^{\op} \rightarrow V$ sends weighted colimits in $X$ to weighted limits. This means that $a(\colim(j,f), x) = \limit(j^{\circ}, a(-,x)), $ for all fitting $j$ and $f$; where $j^{\circ}$ is the distributor obtained from $j$ by applying the involution $(-)^{\circ}$.
		\item Right adjoints are continuous; dually, left adjoints are cocontinuous.
 	\end{enumerate}
\end{Exs}
There is a nice characterization of (co)complete $V$-categories in terms of adjunctions.
\begin{teorema}\label{Coco}
	Let $(X,a)$ be a $V$-category. Then $(X,a)$ is complete iff the co-Yoneda embedding has a right adjoint; dually, it is cocomplete if the Yoneda embedding has a left adjoint. Moreover, in case they exist, they have the following form:
	$$ \limit : \mathbb{U}(X) \rightarrow X, \ \  l \mapsto \limit(l,\Id),$$
	$$ \colim : \mathbb{D}(X) \rightarrow X, \ \ j \mapsto \colim(j,\Id).$$
\end{teorema}
\begin{re}
	Since both $\mathbb{U}(-)$ and $\mathbb{D}(-)$ define  \textit{Kock–Zöberlein} monads (see \cite{KOCK199541}) on $\Vcat$, in order to prove that the (co)-Yoneda embedding has a (right)left adjoint, it is sufficient to provide a left inverse to it.
\end{re}
\section{Quantale-Enriched Multicategories}
$(L,V)$-categories are a special case of the more general $(T,V)$-categories, where the list monad $L$ is considered. They are also the order-enriched version of \textit{multicategories} (see \cite{leinster2004higher} for an account on them, and \cite{Lambek} for a historical perspective). The basic idea is that, instead of having arrows with just a single object as the domain, we allow them to have as domain a list of objects.\par\medskip

In this section we introduce $(L,V)$-categories and some of their basic constructions, by mirroring what we have done in the previous section.
\subsection{(L,V)-Categories and (L,V)-functors}
Recall that the list monad is defined as the monad whose underlying functor is given by
$$L : \mathtt{Sets} \rightarrow \mathtt{Sets}, \ \ f: X \rightarrow Y \mapsto Lf : \amalg_{n \geq 0} X^n \rightarrow \amalg_{m \geq 0} Y^m , \ \ \underline{x} =(x_1,..., x_n) \mapsto (f(x_1),...,f(x_n)),$$
and whose unit and multiplication at a set $X$ are defined as:
\begin{itemize}
	\item $e_X : X \rightarrow L(X),  \ \ x \mapsto (x);$
	\item $m_X : L^2(X) \rightarrow L(X), \ \ (\underline{x}_1,...,\underline{x}_n) \mapsto (x_{11}, ..., x_{1k}, ..., x_{n1}, ...,x_{nl}).$
\end{itemize}

\begin{re}
	Let $\underline{x}$, $\underline{w}$ be two lists; in order to avoid possible confusion with the list of lists $\underline{\underline{y}} = (\underline{x},\underline{w})$, we denote the list obtained by concatenating $\underline{x}$ and $\underline{w}$ as $(\underline{x};\underline{w})$. Moreover, in the case in which one of the two is a single element list, we use the shortcut $(\underline{x};w)$ instead of $(\underline{x};(w)).$
\end{re}

We can extend (in a functorial way) the list monad $L$ to $\Mat{V}$ by defining, for $r: X \xslashedrightarrow{} Y$:
$$\tilde{L}r : L(X)  \xslashedrightarrow{} L(Y),  \ \ (\underline{x}, \underline{y}) \mapsto \begin{cases}
r(x_1,y_1) \otimes ... \otimes r(x_n, y_n) \mbox{ if the two lists have the same length,} \\
\perp \mbox{ otherwise.} \\
\end{cases} $$
One can prove that this particular extension defines a monad on $\Mat{V}$ that, moreover, preserves the involution
$$(-)^{\circ} : \Mat{V}^{\op} \rightarrow \Mat{V}.$$
\begin{re}
	From now on we will use $L$ for both the ordinary list monad and its extension to $\Mat{V}$.
\end{re}
This allows us to define the order-enriched category $\Mat{(L,V)}$ in which a morphism $r: X \kmodto Y$ is a $V$-relation of the form
$$ r : L(X) \xslashedrightarrow{} Y,$$
and in which composition is given by
$$ s \circ r = s \cdot Lr \cdot m_X^{\circ},$$
where $e_X^{\circ} : X \kmodto X$ is the identity.
\begin{re}\label{K}
	Note that, due to the Kleisli-style composition we defined, $ - \circ r$ preserves suprema, but $s \circ (=)$ does not in general.
\end{re}
\begin{defin}
		An $(L,V)\mbox{-}$category is a pair $(X, a)$, where $X$ is a set and $a : X \kmodto X$ is an $(L,V)$-relation that satisfies:
	\begin{itemize}
		\item $e_X^{\circ} \leq a$;
		\item $a \circ a \leq a.$
	\end{itemize}
\end{defin}
\begin{re}
	When $V = \mathbf{2}$, the $(L,\mathbf{2})$-structure of an $(L,\mathbf{2})$-category $(X,a)$ is a subset $a \subseteq L(X) \times X$ such that:
	\begin{itemize}
		\item for all $x \in X$, $ ((x),x) \in a$;
		\item given $(\underline{z}_1,...,\underline{z}_n) \in L^2(X)$, $\underline{x} \in LX$, and $y \in X$, such that
		$$((\underline{z}_1,...,\underline{z}_n),\underline{x} )\in La, \mbox{ and } (\underline{x}, y) \in a,$$
		then
		$$((\underline{z}_1;...;\underline{z}_n),y) \in a.$$
	\end{itemize}
	Notice that $La$ is the subset that corresponds to the relation $La : L^2(X) \times LX \rightarrow \mathbf{2}$ which one obtains by applying the extension of the list monad to the relation $a : LX \times X \rightarrow \mathbf{2}$.
\end{re}
\begin{defin}
	Given two $(L,V)\mbox{-}$categories $(X,a), (Y,b)$, an \textit{$(L,V)\mbox{-}$functor} $f: (X,a) \rightarrow (Y,b)$ is a function between the underlying sets such that
	$$ a \leq f^{\circ}\cdot b \cdot Lf$$
	which, in pointwise terms, means that, for all $\underline{x} \in LX$, $y\in X$,
	$$ a(\underline{x},y) \leq b(Lf(\underline{x}), f(y)).$$
	If the equality holds, we call $f$ fully faithful.
\end{defin}
\begin{re}
If $V=\mathbf{2}$, then an $(L,\mathbf{2})$-functor $f: (X,a) \rightarrow (Y,b)$ satisfies, for all $\underline{x} \in LX$, $y\in X$,
$$ (\underline{x},y) \in a \mbox{ implies } (Lf(\underline{x}), f(y)) \in b.$$
Notice how this generalizes the classical monotonicity condition.
\end{re}
In this way we define $\Multi{V}$ as the category whose objects are $(L,V)$-categories and whose arrows are $(L,V)$-functors, moreover, $\Multi{V}$ becomes an order-enriched category by defining, for two $(L,V)$-functors $f,g : (X,a) \rightarrow (Y,b)$,
$$ f \leq g \mbox{  iff }  k \leq \bigwedge_{x \in X} b(Lf((x) ),g(x)).$$
\begin{Exs}
	\begin{enumerate}
		\item Every set $X$ defines an $(L,V)$-category with $e^{\circ}_X$ as $(L,V)$-structure. In particular, we can define the one-point $(L,V)$-category $E  = (1, e^{\circ}_{1} ).$
		\item Every set $X$ defines an $(L,V)$-category if we consider the free $L$-algebra on $X$, $(LX, m_X)$.
		\item $V$ itself defines an $(L,V)$-category where $[\underline{v}, w] = [v_1 \otimes ... \otimes v_n, w].$
		\item For two $(L,V)\mbox{-}$categories $(X,a), (Y,b)$, we can form their tensor product $X \boxtimes Y  = (X \times Y, a \boxtimes b),$ where
		$$ a \boxtimes b (\gamma, (x,y)) = a(L\pi_1(\gamma), x) \otimes b(L\pi_2(\gamma), y);$$
		here $\gamma \in L(X \times Y)$ and $\pi_1 ,\pi_2$ are the obvious projections. Unfortunately, in general it is not true that $X \boxtimes E \simeq X.$
	\end{enumerate}
\end{Exs}
	In general, every monoidal $V$-category $(X,a, \mult, u_X)$ defines an $(L,V)$-category, where $a(\underline{x}, y) = a(x_1 \mult ... \mult x_n,y).$\footnote{In particular $a((-), y) = a(u_X, y)$.} This motivates the following definition.
\begin{defin}\label{Repre}
	Let $(X,\hat{a})$ be an $(L,V)$-category. We call $X$ \textit{representable} if there exists a monoidal $V$-category $(X,a, \mult, u_X)$ such that $ \hat{a} = a \cdot \alpha$, where $$\alpha : L(X) \rightarrow X, \ \ \underline{x} \mapsto x_1 \mult ... \mult x_n.$$
\end{defin}
\begin{re}\label{Mon}
	Definition \ref{Repre} allows us to define a $2$-functor $ \mathtt{Kmp} :\Vcat^L \rightarrow \Multi{V}$ which has a left adjoint $M : \Multi{V} \rightarrow \Vcat$ that sends an $(L,V)$-category $(X,a)$ to $(LX, La \cdot m_X^{\circ}, m_X)$ and an $(L,V)$-functor $f$ to $Lf$. $M$ is also a $2$-functor.\\
	Using the aforementioned adjunction, we can extend the monad $L$ to a monad on $\Multi{V}$, denoted by $L$ as well. Moreover, one can prove (see \cite{Chikhladze2015}) that there is an equivalence
	$$\Vcat^L \simeq \Multi{V}^L.$$
\end{re}
\begin{re}
A priori, due to the non-symmetric form of arrows in $\Mat{(L,V)}$, it is not clear how to define an $(L,V)$-category that seems to play the role of a dual. Luckily, we can use the adjunction $\mathtt{Kmp} \dashv M$ and the involution in $\Mat{V}$ to define, for an $(L,V)$-category $(X,a)$, its opposite category as $X^{\op} =(LX, m_X \cdot L a^{\circ} \cdot m_X)$. At first this might be seen as an \textit{ad hoc} definition, but if we apply this construction to a $V$-category $(X,a)$, seen as an $(L,V)$-category $(LX,e_X^{\circ} \cdot a)$, we get
	$$X^{\op} = \mathtt{Kmp}(LX, L a^{\circ}),$$
	where $(LX, L a^{\circ})$ is the dual, as a $V$-category, of $(LX, La)$.
\end{re}
For any $(L,V)$-category $(X,a)$, we can form the $(L,V)$-category $\mathbb{D}_L(X)[-,=]$ whose underlying set consists of all $(L,V)$-functors of the form: $f : X^{\op} \boxtimes E\rightarrow V$ and whose $(L,V)$-structure is given by \label{DLstructure}
$$ \mathbb{D}_L(X)[\underline{f},g] =  \bigwedge_{(\underline{x}_1, ...,\underline{x}_n) \in LX^2 } [(f_1(\underline{x}_1),..., f_n(\underline{x}_n)), g(m_x((\underline{x}_1, ...,\underline{x}_n)))],$$
where $\underline{f} \in L( \mathbb{D}_L(X)) $ and $g \in  \mathbb{D}_L(X).$
\begin{re}\label{MultiYoneda}
We have a fully faithful functor, called the Yoneda embedding,
$$
\Yo{X} : X \rightarrow \mathbb{D}_L(X), \ \ x \mapsto a(-, x).
$$
Moreover, it can be proved that
$$ \mathbb{D}_L(X)[L\Yo{X}(\underline{x}),g] = g(\underline{x}).$$
The last result is known as the Yoneda Lemma.
\end{re}
\subsection{Distributors and Weighted Colimits in $\Multi{V}$}
The relational point of view we used for introducing $V$-categories allows us to provide the corresponding notion of distributor for $(L,V)$-categories by considering the composition $\circ$ defined in the previous section. Unfortunately, due to the non-symmetric form of this composition, we won't be able to provide a theory of weighted limits but only one for weighted colimits.
\begin{defin}\cite{CH09}
	Given two $(L,V)$-categories $(X,a), (Y,b)$, an \textit{$(L,V)$-distributor} $j  : (X,a) \kmodto (Y,b)$ is an $(L,V)$-relation between the underlying sets such that:
	\begin{itemize}
		\item $j \circ a \leq j;$
		\item $b \circ  j \leq j.$
	\end{itemize}
\end{defin}
Just as in the $V$-case, we can define an order-enriched category $\Dist{(L,V)}$, where the composition is the one defined in $\Mat{(L,V)}$.
\begin{re}\label{mate}
	As in the $V$-case, one can prove
	$$\Dist{(L,V)}(X,Y) \simeq \Multi{V}(X^{\op} \boxtimes Y,V)\simeq \Multi{V}(Y,\mathbb{D}_L(X)).$$
	In particular, to every $(L,V)$-distributor $j : X \kmodto Y$ we can associate its \textit{mate}
	$$ \ulcorner j \urcorner : Y \rightarrow \mathbb{D}(X), \ \ y \mapsto j(-,y).$$
\end{re}
Just as in the $V$-case, we have the equivalent of Theorem \ref{Pro}; where now, for an $(L,V)\mbox{-}$functor $f: (X,a) \rightarrow (Y,b)$, the two associated $(L,V)$-distributors are:
\begin{enumerate}
	\item $f_{\circledast} : X \kmodto Y, \ \ f_{\circledast}(\underline{x},y) = b(Lf(\underline{x}), y);$
	\item $f^{\circledast} : Y \kmodto X, \ \ f^{\circledast}(\underline{y},x) = b(\underline{y},f(x)).$
\end{enumerate}
As we have already noticed in Remark \ref{K}, due to the intrinsic asymmetry of the Kleisli composition, in contrast to the $V$-case, we can only prove that, for every $(L,V)$-distributor $\alpha : Z \kmodto X,$
\begin{alignat*}{2}
(-) \circ \alpha  : \Dist{(L,V)}(X,Y) &\rightarrow \Dist{(L,V)}(Z,Y) \\
\beta        &\longmapsto \beta\circ \alpha
\end{alignat*}
has a right adjoint
\begin{alignat*}{2}
(-) \blacktriangleright \alpha  : \Dist{(L,V)}(Z,Y) &\rightarrow \Dist{(L,V)}(X,Y) \\
\beta        &\longmapsto \beta \blacktriangleright \alpha = \beta \rhd \hat{\alpha}
\end{alignat*}
where $\hat{\alpha} = L\alpha \cdot m_X^{\circ}.$\\
\begin{defin}
	Let $f : (X,a) \rightarrow (Y,b)$ be an $(L,V)$-functor and let $j : X \kmodto Z$ be an $(L,V)$-distributor. We say that an $(L,V)$-functor $g : (Z,c) \rightarrow (Y,b)$ is the \textit{colimit of $f$ weighted by $j$}, and denote it by $\Tcolim{L}(j,f)$, if $g_{\circledast} \simeq f_{\circledast} \blacktriangleright  j$, that is to say
	$$g_{\circledast}(\underline{z},y) = b(Lg(\underline{z}), y) = \bigwedge_{\underline{x} \in LX} [\hat{j}(\underline{x},z), f_{\circledast}(\underline{x},y)],\ \  (1)$$
	An $(L,V)$-category $(Y,b)$ is called \textit{cocomplete} if it has all weighted colimits; this means that, for all $(L,V)$-functors $f : (X,a) \rightarrow (Y,b)$ and for all $(L,V)$-distributors $ h : Z \kmodto X$, the colimit of $f$ weighted by $h$ exists.
\end{defin}
In contrast to what happens in $\Vcat$, in $\Multi{V}$ cocompleteness cannot be reduced to distributors into $K$ (a.k.a. presheaves). Nevertheless, the analogue of Theorem \ref{Coco} still holds also in the $(L,V)$-case (see \cite[Theorem~2.7]{HOFMANN2011283}).
\begin{Exs}
	\begin{enumerate}
		Let $(X,a)$ be an $(L,V)$-category.
		\item One can prove that every element of $ \mathbb{D}_L(X)$ can be written as the colimit of $\Yo{X}$ weighted by itself.
		\item Consider the constant $(L,V)$-relation: $u : E \kmodto E$, where $u \in V$. Let $x \in X$, $\Tcolim{L}(u \cdot a(-,x),\Id)(1)$ (if it exists) coincides with the copower $x \odot u$ of the underlying $V$-category $(X, e_X \cdot a).$
	\end{enumerate}
\end{Exs}
When the target category is representable, we have the following proposition.
\begin{prop}\label{LCol}
	Let $f : (X,a) \rightarrow (Y,b \cdot \alpha)$ be an $(L,V)$-functor, with $(Y, b \cdot \alpha)$ representable, and let $j : X \kmodto E$ be an $(L,V)$-distributor. If $\Tcolim{L}(j,f)$ exists, then
	$$ \Tcolim{L}(j,f)(1) \simeq \colim(e_1^{\circ} \cdot\hat{j} , \alpha Lf),$$
	where $\alpha Lf : \mathtt{Kmp}(X) \rightarrow (Y,b).$
\end{prop}
Before we prove this proposition we state a useful lemma.
\begin{lem}
		Let $f : (X,a) \rightarrow (Y,b \cdot \alpha)$ be an $(L,V)$-functor, with $(Y, b \cdot \alpha)$ representable, and let $j : X \kmodto E$ be an $(L,V)$-distributor. Then $\alpha Lf : \mathtt{Kmp}(X) \rightarrow (Y,b)$ is a $V$-functor and $e_1^{\circ} \cdot\hat{j}$ is a $V$-distributor.
\end{lem}
\begin{proof}
	Since $f$ is an $(L,V)$-functor, $f_{\circledast } = b \cdot \alpha Lf = (\alpha Lf)_{*}$ is an $(L,V)$-distributor, hence
	$$ f_{\circledast } \cdot La \cdot m_X^{\circ} \leq f_{\circledast };$$
	moreover, because $ b \cdot b \leq b$, we have also
	$$ b \cdot  b \cdot \alpha Lf \leq b \cdot \alpha Lf,$$
	from which it follows that $(\alpha Lf)_{*}$ is a $V$-distributor, hence, from Corollary \ref{DistFun}, it follows that $\alpha Lf $ is a $V$-functor. \\
	A quick calculation shows that $e_1^{\circ} \cdot \hat{j} = j$; moreover, since $j$ is an $(L,V)$-distributor, we have
	$$ e_1^{\circ} \cdot \hat{j} \cdot La \cdot m_X^{\circ} = j \cdot La \cdot m_X^{\circ} = j \circ a \leq j.$$
	This shows that $e_1^{\circ} \cdot \hat{j} : LX \kmodto E$ is a $V$-distributor. \\
\end{proof}
\begin{proof}(Of the proposition)
	We have
	\begin{alignat*}{2}
	 \tilde{b}( e_Y(\Tcolim{L}(j,f)(1)),c)& = b(\Tcolim{L}(j,f)(1) ,c) \\
	    & = \bigwedge_{\underline{x} \in LX} [\hat{j}(\underline{x},e_1(1)), b(\alpha Lf(\underline{x}),c)] \\
	    &= \bigwedge_{\underline{x} \in LX} [j(\underline{x}), b(\alpha Lf(\underline{x}),c)] \\
	    & =  b(\colim(e_1^{\circ} \cdot\hat{j} , \alpha Lf), c).
	\end{alignat*}
\end{proof}
\begin{cor}
	Let $f : (X,a) \rightarrow (Y,b \cdot \alpha)$ be an $(L,V)$-functor, with $(Y, b \cdot \alpha)$ representable, and let $j : X \kmodto E$ be an $(L,V)$-distributor. Then
	$$\tilde{b}(e_Y(\Tcolim{L}(j,f)(1)) , c) = \limit(e_1^{\circ} \cdot\hat{j}, b(\alpha Lf(-) ,c)).$$
\end{cor}
\begin{proof}
	The assertion follows from Example \ref{Inv} of Examples \ref{ColLims}. \\
\end{proof}
\begin{cor}\label{Deco}
	Let $f : (X,a) \rightarrow (Y,b \cdot \alpha)$ be an $(L,V)$-functor, with $(Y, b \cdot \alpha)$ representable, and let $j : X \kmodto E$ be an $(L,V)$-distributor. Then (if it exists) $\Tcolim{L}(j,f)(1) \simeq \bigvee_{\underline{x} \in LX} j(\underline{x}) \odot \alpha Lf(\underline{x}).$ Where the conical colimit and the copower are taken in the underlying $V$-category $(Y,b)$.
\end{cor}
\begin{proof}
	The assertion follows from the discussion on weighted colimits done in Remark \ref{Deco1}. \\
\end{proof}
\begin{defin}
	Let $f : (X,a) \rightarrow (Y,b)$ and $g : (Y,b) \rightarrow (X,a)$ be two $(L,V)$-functors; $f$ is \textit{left adjoint} to $g$ (vice versa $g$ is \textit{right adjoint} to $f$) and denoted by $f \dashv g$, if $f_{\circledast} \simeq g^{\circledast}$. In pointwise terms, this means that, for all $\underline{x},y \in LX,Y,$
	$$ b(Lf(\underline{x}),y) = a(\underline{x},g(y)).$$
\end{defin}
\begin{Exs}
	\begin{enumerate}
		\item Every $(L,V)$-functor $f : (X,a) \rightarrow (Y,b)$ defines an adjoint pair of $(L,V)$-functors
		\[
		\begin{tikzcd}
		\mathbb{D}_L(X) \ar[r,bend left,"\mathbb{D}(f) = - \circ f^{\circledast}",""{name=A, below}] & \mathbb{D}_L(Y). \ar[l,bend left,"- \circ f_{\circledast}",""{name=B,above}] \ar[from=A, to=B, symbol=\dashv]
		\end{tikzcd}
		\]
		\item Given an $(L,V)$-functor $f : (X,a) \rightarrow (Y,b)$, with $(Y,b)$ cocomplete, we can consider the following functor, where $\Yo{X} : X \rightarrow \mathbb{D}_L(X)$ is the Yoneda embedding we introduced in Remark \ref{MultiYoneda},
		$$ \mathtt{Lan}_y(f) = \Tcolim{L}(y_{\circledast}, f)(1) : \mathbb{D}_L(X)  \rightarrow Y,$$
		which is called the \textit{left Kan extension of $f$ along $\Yo{X}$}. One can show that $\mathtt{Lan}_y(f)$ has a right adjoint $R$ given by
		$$ R(y) = b(Lf(-),y) : Y \rightarrow \mathbb{D}_L(X).$$
	\end{enumerate}
\end{Exs}
Let $f: (X,a)\rightarrow (Y,b)$ and $g: (Y,b) \rightarrow (Z,c)$ be two $(L,V)$-functors and let $j : X\kmodto D$ be a distributor. From the definition of weighted colimit it follows that
$\Tcolim{L}(j, gf) \leq g(\Tcolim{L}(j,f))$. \\
If the equality holds, we say that $g$ preserves the colimit of $f$ weighted by $h$. If $g$ preserves all colimits, we say that $g$ is \textit{cocontinuous}. As in the $V$-case, we have that left adjoint $(L,V)$-functors are cocontinuous.
\section{(Enriched)-Quantales and Lax Monoidal Monads}
In this section we define the enriched counterparts of quantales, we state some of their properties and we introduce lax monoidal monads.\par\medskip
It is well known that quantales are monoids in the monoidal category $\mathtt{Sup}$ of sup-lattices, where the monoidal structure on $\mathtt{Sup}$ is the one induced by the \textit{strong-commutative power set monad} (see \cite{EGTG} for details).
Since $\mathtt{Sup}$ is equivalent to the category of (separated) cocomplete, for the quantale $\mathbf{2}=\{0,1\}$, $\mathbf{2}$-categories (with cocontinuous $\mathbf{2}$-functors as arrows), it is natural to define a $V$-quantale as a monoid in the category of cocomplete $V$-categories, where the monoidal structure is the one induced by the strong commutative monad $(\Pow{V}, u, n)$ (the $V$-powerset monad). Here
$\Pow{V} : \Sets \rightarrow \Sets$ is defined by $\Pow{V}(X) =  V^X$ and, for $f: X \rightarrow Y$ and $\phi \in V^X$,
$$ \Pow{V}(f)(\phi)(y) =  \bigvee_{x \in f^{-1}(y) }\phi(x).$$
Moreover:
\begin{itemize}
	\item the unit $u_X : X \rightarrow V^X$ is the transpose of the diagonal $\bigtriangleup_X : X \times X \rightarrow V$;
	\item the multiplication $n_X : \Pow{V}(\Pow{V}(X)) \rightarrow \Pow{V}(X)$ is defined by $n_X(\Phi)(x) =  \bigvee_{\phi \in V^X} \Phi(\phi) \otimes \phi(x)$.
\end{itemize}
This motivates the following definition.
\begin{defin}
	A $V$-quantale $(Q, \mult, k)$ is a monoid in the monoidal category $\Alg{V}$.
\end{defin}
\begin{re}\label{LimpRimp}
Just as $\mathtt{Sup}$ is equivalent to the category of separated cocomplete $\mathbf{2}$-categories, $\Alg{V}$, the $V$-enriched version of $\mathtt{Sup}$, is equivalent to the category of separated cocomplete $V$-categories with cocontinuous $V$-functors. Thus, a $V$-quantale $(Q, \mult, k)$ is a cocomplete monoidal $V$-category such that the multiplication $\mult$ is cocontinuous in each variable. That is to say, given an element $x \in Q$, $ - \mult x : Q \rightarrow Q$ and $ x \mult = : Q \rightarrow Q$ are two cocontinuous $V$-functors. Since $Q$ is cocomplete, it follows that both functors have right adjoints which are given by
	$$  y  \rimp (=) =  \colim(Q(- \mult y, =) , \Id), \ \ x \limp (=) = \colim(Q(x \mult =, =) ; \Id),$$
	and which can be re-written in pointwise terms as
	$$ y \rimp z = \bigvee_{x\in X} Q(x\mult y,z ) \odot x, \ \ x \limp z = \bigvee_{y\in Y} Q(x \mult y, z) \odot y.$$
	Notice that in the ordered case, the last two expressions correspond to the familiar notion of \textit{left} and \textit{right implication} respectively.
\end{re}
Recall that a $V$-functor $f : (X,a, \mult_X, k_X) \rightarrow (Y,b, \mult_Y, k_Y)$ between monoidal $V$-categories is called \textit{lax monoidal} if
$$  k_Y \leq f(k_X), \ \ f(x) \mult_Y f(z) \leq f(x \mult_X z).$$
Let $\mathtt{Mon}(V \mbox{-}\mathtt{Cat})_{{\mathtt{lax}}}$ be the category formed by monoidal $V$-categories with lax monoidal functors between them.
\begin{re}
Notice that $\mathtt{Mon}(\mathbf{2} \mbox{-}\mathtt{Cat})_{{\mathtt{lax}}}$ is related to the category of ordered monoids and submultiplicative order-preserving maps considered in \cite{InjHulls}. The only difference is that here we consider ordered monoids whose underlying ordered sets are not separated (i.e. they are not partially ordered).
\end{re}
\begin{prop} \label{Fully-Faithful}
	$\mathtt{Kmp} : \mathtt{Mon}(V \mbox{-}\mathtt{Cat})_{\mathtt{lax}} \rightarrow \Multi{V}$ is fully faithful. Here $\mathtt{Kmp}$ is the functor defined as in Remark \ref{Mon}.
\end{prop}

\begin{proof}
	The fact that $\mathtt{Kmp}$ is faithful is straightforward.\\
	Let $f : \mathtt{Kmp}(X) \rightarrow \mathtt{Kmp}(Y)$ be an $(L,V)$-functor and let $\tilde{a}$, $\tilde{b}$ be the $(L,V)$-structures on $\mathtt{Kmp}(X)$ and $\mathtt{Kmp}(Y)$, respectively. By definition we have that
	$$\tilde{a}(\underline{x}, y) \leq \tilde{b}(Lf(\underline{x}),f(y)), \ \ \mbox{ where } Lf(\underline{x}) = (f(x_1), ... ,f(x_n)),$$
	which by definition implies
	$$ a(x_1 \mult_X ...  \mult_X x_n, y ) \leq b(f(x_1)\mult_Y... \mult_Y f(x_n), f(y)).$$
	If we take $y = x_1 \mult_X ...  \mult_X x_n$, from $ k \leq a(x,x)$ it follows that
	$$f(x) \mult_Y f(y) \leq f(x \mult_X y).$$
	Finally, if we take as first argument the empty list, we have
	$$ k_Y \leq f(k_X).$$
\end{proof}
\begin{re}
When $V = \mathbf{2}$, this proposition sheds light on the definition of embedding between ordered monoids contained in \cite{InjHulls}. There, a submultiplicative map (lax monoidal in our notation) $f : X \rightarrow Y$ is called an embedding if, for all $x_1,...,x_n,x  \in X$,
$$ f(x_1),...,f(x_n) \leq f(x) \mbox{ implies } x_1,...,x_n \leq x.$$
This condition corresponds to $f$ being fully faithful as an $(L,V)$-functor. This suggests that $\Multi{V}$ provides the common roof we were searching for.
\end{re}
By relaxing the definition of $V$-quantale we can define another useful category to which the ideas will apply that we are going to develop in the following sections.
\begin{defin}
	A monoidal $V$-category $(X,a, \mult_X, k_X)$ is called \textit{residuated} if, for every $x \in X,$ $ x \mult_X (=)$ and $(-)\mult_X x$ have right adjoints (in $\Vcat$) $x \rimp(=)$ and $x \limp (=)$, respectively. The category formed by residuated monoidal $V$-categories with lax monoidal functors between them is denoted by $\mathtt{ResMon}(V \mbox{-}\mathtt{Cat})_{\mathtt{lax}}. $
\end{defin}
\begin{re}
When $V=\mathbf{2}$, a residuated monoidal category is a residuated monoid. These are closely related to residuated lattices, except that we do not require the existence of finite products/coproducts (that is to say, we do not require our order relation to come from a lattice).
\end{re}
\begin{re}
	From the definition it is clear that $V$-quantales form a full subcategory of both $\mathtt{ResMon}(V \mbox{-}\mathtt{Cat})_{\mathtt{lax}} $ and $\mathtt{Mon}(V \mbox{-}\mathtt{Cat})_{\mathtt{lax}}$.
\end{re}
\begin{re}
	Notice that a morphism $f$ in $\mathtt{ResMon}(V \mbox{-}\mathtt{Cat})_{\mathtt{lax}} $ satisfies
	$$ f(y \rimp z) \leq f(y) \rimp f(z), \ \  f(y \limp z) \leq f(y) \limp f(z).$$
	Indeed, from
	$$ f(y \rimp z)  \mult f(y) \leq f((y \rimp z )\mult y) \mbox{ and } (y \rimp z )\mult y \leq z$$
	it follows that
	$$f(y \rimp z) \leq f(y) \rimp f(z).$$
	In the same way, but starting from
	$$  f(y) \mult  f(y \limp z) \leq f( y \mult (y \limp z)),$$
	we get
	$$ f(y \limp z) \leq f(y) \limp f(z).$$
\end{re}
\begin{defin}
	Given a $V$-quantale $Q$, a \textit{lax monoidal monad} $T: Q \rightarrow Q$ is a lax monoidal functor such that:
	\begin{itemize}
		\item  $k_Q \leq T(k_Q);$
		\item $T^2 = T.$
	\end{itemize}
\end{defin}
\begin{teorema}
	Let $Q$ be a $V$-quantale and $T : Q \rightarrow Q$ be lax monoidal monad. Then
	$$Q^T = \{x \in Q \mbox{ such that } T(x) = x \},$$
	is a $V$-quantale; moreover the inclusion $Q^T \rightarrow Q$ is a $V$-functor that has a left adjoint $\pi_T$ which is a morphism of $V$-quantales.
\end{teorema}
\begin{proof}
	The proof is a straightforward generalization of the corresponding proof for quantic nuclei. The $V$-structure on $Q^T$ is the one induced by $Q$ while the right adjoint to the inclusion functor is $\pi_T(x) =T(x)$.\\
\end{proof}
\section{Intermezzo}\label{Intermezzo}
In this section we recall how injectives and injective hulls are built in the category of partially ordered sets and monotone maps between them. We show how the same ideas apply to the enriched case, for a general commutative quantale $V$.\par\medskip
Recall that in a category $C$, an object $X$ is called \textit{injective} (with respect to a class of morphisms $J$), if for every diagram of the form
\[
\begin{tikzcd}
Y \arrow[r] \arrow[d] & X \\
Z
\end{tikzcd}
\]
where $Y \rightarrow Z$ is in $J$, there exists an extension
\[
\begin{tikzcd}
Y \arrow[r] \arrow[d]& X \\
Z \arrow[ur, dotted]
\end{tikzcd}
\]
that makes the diagram commute.\\
Given an object $Y$, $E$ is called an \textit{injective hull} for $Y$ (with respect to a class of morphisms $J$) if there exists an arrow $e : Y \rightarrow E$ in $J$ such that:
\begin{itemize}
	\item $E$ is $J$-injective;
	\item $e$ is \textit{essential}: for every $f : E \rightarrow X$,  $fe \in J$ implies $f \in J$.
\end{itemize}
Let $(X,\leq)$ be a partially ordered set. Recall that we have two embeddings:
$$\downarrow :  X \rightarrow \mathbb{D}(X), \ \ x \mapsto \downarrow x =\{w \in X, \ \ w\leq x \},$$
$$\uparrow :  X \rightarrow \mathbb{U}(X), \ \ x\mapsto \uparrow x =\{w\in X, \ \ x\leq w\}.$$
It is well known that there exists an adjunction
\[
\begin{tikzcd}
\mathbb{D}(X) \ar[r,bend left,"L",""{name=A, below}] & \mathbb{U}(X), \ar[l,bend left,"R",""{name=B,above}] \ar[from=A, to=B, symbol=\dashv]
\end{tikzcd}
\]
where, for $j\in \mathbb{D}(X)$, $l \in \mathbb{U}(X)$, one has
$$L(j)= \{x \in X \mbox{ : } \forall_{w \in X} w\in j \Rightarrow w \leq x\}, \ \ R(l)= \{x \in X \mbox{ : } \forall_{w \in X} w\in l \Rightarrow x\leq w\}.$$
This adjunction induces a monad (closure operator) $T$ on $\mathbb{D}(X)$, defined, for $j\in \mathbb{D}(X)$, as follows:
\begin{alignat*}{2}
T(j) & = \{x \in X \mbox{ : } \forall_{y \in X} y\in L(j) \Rightarrow x\leq y\} \\
& = \{x \in X \mbox{ : } \forall_{y \in X} (\forall_{w \in X} w\in j \Rightarrow w \leq y ) \Rightarrow x\leq y\} \\
& = \bigwedge_{\{y \mbox{ : }\forall_{w \in X} w\in j \Rightarrow w \leq y\}  } \downarrow y\\
& = \bigwedge_{\{y \mbox{ : } j \leq \downarrow y\}  } \downarrow y.
\end{alignat*}
In particular, it is immediate from the definition that, for all $x \in X$,
$$T(\downarrow x) =  \bigwedge_{\{y \mbox{ : } \downarrow x \leq \downarrow y  \}}\downarrow y = \downarrow x. \ \ $$
Thus we have the following commutative diagram
\begin{center}
	\begin{tikzcd}
	X \arrow[dr, bend right, "\downarrow(-)" ] \arrow[r, "\tilde{\downarrow}(-)" ] & \mathbb{D}(X)^T \\
	& \mathbb{D}(X) \arrow[u, "\pi_T"]
	\end{tikzcd}
\end{center}
The set of fixed points of this monad forms a partially ordered set denoted by $\mathbb{D}(X)^T$ which is called the \textit{Dedekind-MacNeille completion} of $X$. The term "completion" refers to the following property of $\mathbb{D}(X)^T$: every element of $\mathbb{D}(X)^T$ is both a supremum of elements of $X$ and an infimum of elements of $X$, and the embedding $ X \rightarrow \mathbb{D}(X)^T$ is the smallest embedding of $X$ in a complete ordered set with this property. Moreover, in \cite{MacNeille1937}, MacNeille proved that $\mathbb{D}(X)^T$ is an injective hull for $X$ with respect to monotone embeddings.\par\medskip

This construction admits an "easy" generalization to the enriched case. First of all, we recall that injectives in $\Vcat$ with respect to fully faithful functors are cocomplete $V$-categories, a statement that one can easily prove by using the characterization of cocomplete $V$-categories we mentioned in Theorem \ref{Coco}.
\begin{re}
	Injectives in $\Vcat$ (with respect to fully faithful functors) are exactly the algebras for the Kock-Zöberlein monad $\mathbb{D}(-)$, as it is explained in \cite{Esc97, Esc98, EF99} in a more general context.
\end{re}
Let $(X,a)$ be a $V$-category. As in the ordered case, we have an adjunction, called the \textit{Isbell adjunction}
\[
\begin{tikzcd}
\mathbb{D}(X) \ar[r,bend left,"L",""{name=A, below}] & \mathbb{U}(X), \ar[l,bend left,"R",""{name=B,above}] \ar[from=A, to=B, symbol=\dashv]
\end{tikzcd}
\]
where, for $j\in \mathbb{D}(X)$, $l \in \mathbb{U}(X)$, one has
$$L(j)= \bigwedge_{x \in X} [j(x), a(x,=)] \simeq (\Yo{X})^* \cdot j_*, \ \ R(l)=\bigwedge_{x \in X} [j(x), a(-,x)] \simeq l^* \cdot (\CoYo{X})_* .$$
The monad $T$ induced by this adjunction can be calculated as before:
\begin{alignat*}{2}
T(j) & =\bigwedge_{x \in X} [L(j)(x), a(-,x)] \\
& = \bigwedge_{x \in X} L(j)(x) \pitchfork \Yo{X}(x) \\
& = \limit(L(j),\Yo{X})\\
& = \limit((\Yo{X})^* \cdot j_*, \Yo{X}).
\end{alignat*}
In particular, it is immediate from the definition that, for all $x \in X$,
$$T(\Yo{X}(x)) = \limit((\Yo{X})^* \cdot j_*, \Yo{X}) \simeq \Yo{X}(x). \ \ $$
Thus, as before, the Yoneda embedding factorizes as in
\begin{center}
	\begin{tikzcd}
	X \arrow[dr, bend right, "\Yo{X}" ] \arrow[r, "\tilde{\Yo{X}}" ] & \mathbb{D}(X)^T \\
	& \mathbb{D}(X). \arrow[u, "\pi_T"]
	\end{tikzcd}
\end{center}
The fixed points of this monad form a $V$-category $\mathbb{D}(X)^T$ in which every element $j$ can be written as
$$ j \simeq \colim(j^* \cdot (\Yo{X})_* ,\Yo{X}), \ \ j \simeq \limit((\Yo{X})^* \cdot j_*, \Yo{X}),$$
where the first isomorphism is the well know fact that every presheaf is the weighted colimit of the Yoneda embedding weighted by itself (see \cite{ECT}), and the second one follows from $j$ being a fixed point of $T$.\\
This motivates the following definition.
\begin{defin}\label{DeV}
	A fully faithful $V$-functor $i: (X,a) \rightarrow (Z,b)$ is called \textit{dense} if every element $z \in Z$ is of the form
	$$ z \simeq \colim(z^* \cdot i_* ,i), \ \ z \simeq \limit(i^* \cdot z_*, i).$$
\end{defin}
In this way we have a nice characterization of essential morphisms in $\Vcat$ as contained in the following theorem.
\begin{teorema}\label{Ess}
	A fully faithful $V$ functor $i: (X,a) \rightarrow (Z,b)$ is dense if and only if it is essential.
\end{teorema}
Before we prove it, we state a straightforward lemma.
\begin{lem}
	Suppose we have a fully faithful $V$-functor $f: (Z,b) \rightarrow (Y,c)$ and a $V$-functor $g : (X,a) \rightarrow (Z,b)$. If $fg $ is dense, then $g$ is dense too.
\end{lem}
\begin{proof}(Of the theorem)
	$\Rightarrow).$  Let $f: (Z,b) \rightarrow (Y,c)$ be a $V$-functor such that $\restr{f}{X}$ is fully faithful. We have to show that
	$$c(f(x),f(z)) \leq b(x,z),  \mbox{ for all x,z $\in Z$}.$$
	Since $i$ is dense, we can write
	$$ x \simeq \colim(x^* \cdot i_*, i)), \ \ z \simeq \limit(i^* \cdot z_*,i).$$
	By playing with the properties of (co)limits, we have
	\begin{alignat*}{2}
	 b(x,z)  & = b(\colim(x^* \cdot i_*, i(w)),\limit(i^* \cdot z_*, i(q))) \\
	& = \limit(x^* \cdot i_*, \limit,(i^* \cdot z_*), b(i(w),i(q))))  \\
	& \mbox{(since $i$ is fully faithful)} \\
	& = \limit(x^* \cdot i_*, \limit,(i^* \cdot z_*), a(w,q)))  \\
	& \mbox{(since $\restr{f}{X}$ is fully faithful)} \\
	& = \limit(x^* \cdot i_*, \limit,(i^* \cdot z_*), c(fi(w),fi(q)))) \\
	 &= c(\colim(x^* \cdot i_*, fi(w)),\limit(i^* \cdot z_*, fi(q))) \\
	 & \geq c(f(\colim(x^* \cdot i_*, i(w)), f(\limit(i^* \cdot z_*, i(q))) \\
	 & \geq c(f(x),f(z)).
	\end{alignat*}
	This shows that $i$ is essential.\par\medskip
	$\Leftarrow).$ Since $\mathbb{D}(X)^T$ is injective, there exists a $V$-functor $Z \rightarrow \mathbb{D}(X)^T $ that makes the following diagram
	\[
	\begin{tikzcd}
	Z \arrow[dr, dashrightarrow]
	& X \arrow[d] \arrow[l]\\
	&\mathbb{D}(X)^T
	\end{tikzcd}
	\]
	commute. Because $i$ is essential, then the dashed arrow is fully faithful. The result now follows from the previous lemma.\\
\end{proof}

This theorem ends our \textit{intermezzo} section. In the next section we apply similar ideas in order to construct injective hulls in $\Multi{V}.$ Unfortunately, the procedure is not so smooth as for $V$-categories. \par\medskip

The first inconvenience is due to the fact that, since so far we are not aware of an object equivalent to $ \mathbb{U}(X)$ for an $(L,V)$-category $(X,a)$, we do not have an analogue of the Isbell adjunction and thus we have to define the monad $T$ "manually". Moreover, since we need the category of algebras for this monad to be an $(L,V)$-category, we need $T$ to be a lax monoidal monad. \par\medskip

The second inconvenience follows directly from the first one. In order to prove that $T$ is lax monoidal, we need to focus our attention on a particular subcategory of $\Multi{V}$: the category of \textit{$V$-quantum B-algebras}. For this category we will be able to mimic all the constructions done in this section and build injective hulls as algebras for a lax monoidal monad which resembles the one we introduced before. With a slight modification of Definition \ref{DeV}, where we substitute the $V$-colimit part with an $(L,V)$-colimit, we will be able to prove the equivalent of Theorem \ref{Ess}. Luckily, the restriction to $V$-quantum B-algebras will not prevent us from constructing an injective hull for every $(L,V)$-category. By embedding every $(L,V)$-category in a $V$-quantum B-algebra we will provide a way to construct an injective hull for every $(L,V)$-category.

\section{Injective Hulls of Enriched Quantum B-Algebras}
Promonoidal categories were introduced by Day in his thesis \cite{day_1971}. They originated from the observation that, in order to define a monoidal structure on the functor category $[X,V]$, one only needs a  \textit{promonoidal} structure on $X$. As we are going to describe it, a promonoidal category is a monoid in $\Dist{V}$, with the monoidal structure inherited from the one in $\Vcat$, as described.\par\medskip

Quantum B-algebras were introduced in \cite{RUMP2014759} and further developed in \cite{Rump2013}. The motivation comes from non-commutative algebraic logic, for which they provide a unified semantic, as better explained in \cite{RUMP2014759, Rump2013}. They are ordered sets equipped with two implications that mimic the residuals in a quantale.\par\medskip

Since ordered monoids, and more generally their enriched counterpart, are promonoidal categories for which (part of) the promonoidal structure is representable, it is natural to ask if the same holds for quantum B-algebras. Not surprisingly, the answer is positive, as briefly shown in the last section of \cite{2016THECO}: quantum B-algebras are \textit{representable} promonoidal categories.\par\medskip
In the following section we take this point of view in order to present the enriched version of quantum B-algebras. We show how $(L,V)$-categories provide a common roof for both of them, by showing how some categorical constructions become natural when we consider both as being full subcategories of $\Multi{V}$.\par\medskip

The section is structured as follows:
\begin{itemize}
	\item In the first subsection we introduce promonoidal categories and quantum B-algebras along with some of their properties.
	\item In the second subsection we show how the characterization of injectives in both categories follows easily from the characterization of injectives in $\Multi{V}$. We also generalize some constructions of \cite{2016THECO} to the enriched case. In particular, we show that every enriched quantum B-algebra admits an injective hull.
\end{itemize}
\subsection{Promonoidal $V$-Categories and $V$-Quantum B-Algebras}
In the previous section we briefly explained how $\Vcat$ can be equipped with a monoidal structure, by defining for two $V$-categories $(X,a), (Y,b)$, their tensor product as
$$ X \boxtimes Y = (X\boxtimes Y, a \otimes b).$$
This monoidal structure induces a monoidal structure on $\Dist{V}$, where for $j : X \xslashedrightarrow{} Y, \ \ l : Z \xslashedrightarrow{} W$, one has
$$ j \boxtimes l : X \boxtimes Z \xslashedrightarrow{} Y \boxtimes W, \ \ j \boxtimes l ((x,z),(y,w)) = j(x,y) \otimes l(z,w).$$
With this in mind, it is possible to talk about monoids in $\Dist{V}$.
\begin{defin}\label{PromoDef}
	A \textit{promonoidal $V$-category} is a monoid in $\Dist{V}$. This means that it is a $V$-category $(X,a)$ together with two distributors:
	\begin{itemize}
		\item $ P : X \boxtimes X \xslashedrightarrow{} X$;
		\item $ J : K \xslashedrightarrow{} X$;
	\end{itemize}
such that:
\begin{itemize}
	\item $P \cdot (P \boxtimes \Id) = P \cdot (\Id \boxtimes P)$;
	\item $P \cdot (J \boxtimes \Id ) = \Id, \ \ P \cdot (\Id \boxtimes J) = \Id$.
\end{itemize}
\end{defin}
\begin{re}
	In pointwise notation, the last two conditions mean:
	\begin{itemize}
		\item  $ \bigvee_{x \in X}P(a,b,x)  \otimes P(x,c,d) = \bigvee_{x \in X} P(a,x,d) \otimes P(b,c,x);$
		\item $\bigvee_{z \in X} J(z) \otimes P(z,x,w) = a(x,w), \quad \bigvee_{z \in X} J(z)\otimes P(x,z,w) = a(x,w).$
	\end{itemize}
\end{re}
\begin{re}
	When $V = \mathbf{2}$, a promonoidal ordered set is an ordered set $(X,\leq_X)$ together with:
	\begin{itemize}
		\item A subset $P$ of $X^{\op} \times X^{\op} \times X$ which is downward closed in the first two components and upward closed in the third one;
		\item An upward closed subset $J \subseteq X.$
	\end{itemize}
	The subsets $P$ and $J$ are subject to the conditions expressed in Definition \ref{PromoDef}. Promonoidal ordered sets appear also in logic\textemdash where they are called ternary frames\textemdash as models of substructural logics (see \cite{RESTALL2006289} for further details).
\end{re}
\begin{defin}\label{ProFun}
	Following Day \cite{Day1995}, let $(X,a,P,J)$, $(Y,b, R,U)$ be promonoidal $V$-categories; a \textit{promonoidal $V$-functor} between them is a $V$-functor $f : (X,a) \rightarrow (Y,b)$ such that, for all $x,y,z \in X$,
	$$ P(x,y,z) \leq R(f(x),f(y),f(z)), \ \ J(x) \leq U(f(x)).$$
	If also the reverse inequalities hold, we call $f$ \textit{strong}.
\end{defin}
In this way we can define $\Pro$ as the category whose objects are promonoidal $V$-categories and whose arrows are promonoidal $V$-functors.
\begin{re}\label{ChangeOfBase4}
	In Remark \ref{ChangeOfBase3} we showed that a strong morphism of quantales $f : V \rightarrow Q$ induces a $2$-functor
	$$ \Dist{V} \rightarrow \Dist{Q}.$$
	If we apply this to the morphisms $f : \mathbf{2} \rightarrow [0, \infty]^{\op}$ and $h : [0, \infty]^{\op} \rightarrow \mathbf{2}$, defined in examples \ref{QuantMorph3} and \ref{QuantMorph2} of Examples \ref{QMorph}, since both of them preserve the monoidal structure (as explained in Remark \ref{ChangeOfBase3}), the two $2$-functors
	$$ F : \Dist{\mathbf{2}} \rightarrow \Dist{[0, \infty]^{\op}},$$
	$$ H : \Dist{[0, \infty]^{\op}} \rightarrow \Dist{\mathbf{2}},$$
	induce two $2$-functors
	$$\mathbf{2} \mbox{-} \mathtt{Pro} \rightarrow [0, \infty]^{\op} \mbox{-} \mathtt{Pro},$$
	$$[0, \infty]^{\op} \mbox{-} \mathtt{Pro} \rightarrow \mathbf{2} \mbox{-} \mathtt{Pro}.$$
\end{re}
\begin{Exs}
	\begin{enumerate}
		\item Every monoidal $V$-category $(X,a, \mult_X, u_X)$ defines a promonoidal one as follows:
		\begin{itemize}
			\item $P(x,y,z) = a(x\mult_X y, z);$
			\item $J(x) = a(u_x,x)$.
		\end{itemize}
	\item Every $V$-quantale $(Q, \mult, u_Q)$ defines a promonoidal $V$-category in two (equivalent) ways; the first one is the one described above, while the second one is as follows:
	\begin{itemize}
		\item $P(x,y,z) = Q(x,y\rimp z) = Q(y, x\limp z);$
		\item $J(x) = Q(u_Q, x).$
	\end{itemize}
The equivalence of the two formulations follows by adjunction.
	\end{enumerate}
We will explore further the last example when we will introduce enriched quantum B-algebras.
\end{Exs}
\begin{re} \label{Day}
	As we mentioned in the introduction of the section, a promonoidal structure induces a monoidal structure on the corresponding presheaf category. More precisely, given a promonoidal category $(X,a,P,J)$, we can define a monoidal structure on $(\mathbb{D}(X), \mult_D, J^{\circ})$ as follows:
	$$j \mult_D l (x) = \bigvee_{w,z \in X} P^{\circ}(w,z,x) \otimes j(w)\otimes l(w).$$
In particular, when the promonoidal structure comes from a monoidal one, the previous formula reads as
$$j \mult_D l (x) = \bigvee_{w,z \in X} a(x,w\mult z) \otimes j(w)\otimes l(w).$$
In this case, $(\mathbb{D}(X), \mult_D, J^{\circ})$ becomes a $V$-quantale, where\textemdash as an example\textemdash the right implication is given by
$$l \rimp h (x) = \bigwedge_{y \in X} [l(y), h(x \mult y)].$$
\end{re}
\begin{prop}\label{Fully}
$\Pro$ is a full subcategory of $\Multi{V}$.
\end{prop}
\begin{proof}
Let $(X,a,P,J)$ be a promonoidal category. We can define an $(L,V)$-category $(X, \tilde{a} = \amalg_{n \geq 0} a_n)$ in the following inductive way:
	\begin{itemize}
		\item $a_0 = J;$
		\item $a_1 = a;$
		\item $a_2 = P ;$
		\item $a_n = P \cdot (\Id \boxtimes a_{n-1} ).$
	\end{itemize}
Indeed, $e_X^{\circ} \leq \tilde{a}$ follows directly from the definition, while $ \tilde{a} \circ \tilde{a} \leq \tilde{a}$ follows from the inductive definition of $\tilde{a}$, from the properties of $P$ and $J$, and from the fact that $a$ is a $V$-structure. To explain this better, let $\underline{\underline{x}} = (\underline{x}_1,...,\underline{x}_n) \in LLX$, $\underline{y} = (y_1,...,y_n) \in LX$, and $z\in X$. We want to show that
$$L\tilde{a}(\underline{\underline{x}}, \underline{y}) \otimes \tilde{a}(\underline{y},z) \leq \tilde{a}(m_X(\underline{\underline{x}}),z).$$
If we unravel the definition, this means
$$a_{l(\underline{x}_1)}(\underline{x}_1,y_1) \otimes ... \otimes a_{l(\underline{x}_n)}(\underline{x}_1,y_1) \otimes a_n(\underline{y},z) \leq a_{l(\underline{x}_1) + ... + l(\underline{x}_n)} (m_X(\underline{\underline{x}}),z), \ \ (1)$$
where $l(-)$ denotes the length of a list.\par\medskip

The proof of $(1)$ is a long and frustrating matter of bookkeeping which relies on the properties that $P$ and $a$ have. Instead of proving $(1)$, we prefer to illustrate the properties used in a simple case which is sufficiently general to justify $(1)$.\par\medskip

Consider $\underline{\underline{x}} = ((x_1,x_2),x_3)$ and $\underline{y}= (y_1,y_2)$. Then $(1)$ becomes
\begin{alignat*}{2}
P(x_1,x_2,y_1) \otimes a(x_3,y_2) \otimes P(y_1,y_2,z) & = \bigvee_{d \in X} P(x_1,x_2,y_1) \otimes a(y_1,d) \otimes a(x_3,y_2) \otimes P(d,y_2,z) \\
& (\mbox{from } a \cdot P \leq P) \\
& \leq \bigvee_{d \in X} P(x_1,x_2,d) \otimes a(x_3,y_2) \otimes P(d,y_2,z) \\
& (\mbox{from } P\cdot (\Id \boxtimes a) \leq P \cdot (a \boxtimes a ) \leq P) \\
&   \leq \bigvee_{d \in X} P(x_1,x_2,d)\otimes P(d,x_3,z) \\
& = \tilde{a}((x_1,x_2,x_3),z).
\end{alignat*}
Suppose that $f : (X,a,P,J) \rightarrow (Y,b,R,U)$ is a promonoidal functor. Then, for all $\underline{x}\in LX$, $y\in X$ (with $l(\underline{x}) = n$),
\begin{alignat*}{2}
\tilde{a}(\underline{x},y) = a_n(\underline{x},y) & = \bigvee_{c \in X} a_{n-1}((x_1,...,x_{n-1}), c) \otimes P(c, x_{n+1},y) \\
& (\mbox{by induction on $a_n$, using the fact that $f$ is a promonoidal functor}) \\
& \leq \bigvee_{c \in X} b_{n-1}((f(x_1),...,f(x_{n-1})), f(c)) \otimes R(f(c), f(x_{n+1}),f(y)) \\
& =b_n(Lf(\underline{x}),f(y)) = b(Lf(\underline{x}),f(y)).
\end{alignat*}
This shows that $f$ is an $(L,V)$-functor.\\
If $\tilde{a}(\underline{x},y) \leq b(Lf(\underline{x}),f(y))$, then
$$\underline{x} = (x_1,x_2) \Rightarrow P(x_1,x_2,y) \leq R(f(x_1),f(x_2),f(y)),$$
$$\underline{x} = (x) \Rightarrow a(x,y) \leq b(f(x),f(y)),$$
$$\underline{x} = (-) \Rightarrow J(x) \leq U(f(x)).$$
This shows that $\Pro$ is a full subcategory of $\Multi{V}$.\\
\end{proof}
\begin{defin}
	A \textit{$V$-quantum B-algebra} $(X,a)$ is a representable promonoidal $V$-category. This means that it is a promonoidal $V$-category $(X,a,P,J)$ equipped with two binary operations $\rimp, \limp : X \times X \rightarrow X$ and an element $u_X \in X$ such that:
	\begin{itemize}
		\item $P(x,-,y) \simeq a(-, x\limp y);$
		\item $P(-,x,y) \simeq a(-, x\rimp y);$
		\item $J(x) = a(u_X,x).$
	\end{itemize}
\end{defin}
\begin{re}
From the representability condition it directly follows that $\rimp, \limp$ are two bi-functors of the form
$$ \rimp, \limp : X^{\op} \boxtimes X \rightarrow X.$$
Moreover, we also have:
\begin{itemize}
	\item $a(x, y \rimp z ) = a(y, x\limp z);$
	\item $x \rimp ( y \limp z) = y \limp (x \rimp  z);$
	\item $u_X\rimp (=)$ ($u_X\limp(=)$) is the identity.
\end{itemize}
In this way we can recover the definition of (unital) ordered quantum B-algebra given in \cite{Rump2013}.
\end{re}
\begin{Exs}
\begin{enumerate}
\item Every $V$-quantale $(Q, \mult, k)$ defines a $V$-quantum B-algebra, where $\rimp$ and $\limp$ are the right adjoints to $\mult$, as explained in Remark \ref{LimpRimp}. In particular, $V$ itself defines two quantum B-algebras. The first one is the ordered quantum B-algebra that comes from the fact that $V$ is an ordered quantale. The second one comes from the fact that the $V$-category $(V, [-,=])$ is a $V$-quantale.
\item In the same way, every residuated monoidal category defines a $V$-quantum B-algebra.
\end{enumerate}
\end{Exs}
\begin{re}
If we combine Remark \ref{ChangeOfBase3} with Remark \ref{ChangeOfBase4}, we can easily show that the two morphisms of quantales $f : \mathbf{2} \rightarrow [0, \infty]^{\op}$ and $h : [0, \infty]^{\op} \rightarrow \mathbf{2}$, defined in examples \ref{QuantMorph3} and \ref{QuantMorph2} of Examples \ref{QMorph}, induce two $2$-functors between the corresponding categories of $V$-quantum B-algebras. Indeed, let $(X,d_X,P, J)$ be a metric quantum B-algebra. From Remark \ref{ChangeOfBase4} it follows that $(X, H(d_X), H(P), H(J))$, where $H(d_X)$, $H(P)$ and $H(J)$ are defined in Remark \ref{ChangeOfBase}, is a promonoidal ordered set. Since, for all $x,y \in X$, one has
$$ H(d_X(-,x)) = H(d_X)(-,x), \ \ H(d_X(x,=)) = H(d_X)(x,=),$$
$$ H(P(x,-,y))= H(P)(x,-,y), \ \ H(P(-,x,y)) = H(P)(x,-,y), \mbox{ and } H(J(x)) = H(J)(x),$$
it follows:
\begin{itemize}
	\item $H(P)(x,-,y) = H(d_X)(-, x\limp y);$
	\item $H(P)(-,x,y) = H(d_X)(-, x\rimp y);$
	\item $H(J)(x) = H(d_X)(u,x).$
\end{itemize}
This shows that $(X, H(d_X), H(P), H(J))$ is an ordered quantum B-algebra.\\
In the same way we can proceed for the $2$-functor defined by $f$.
\end{re}
The definition of promonoidal functor specializes to the following definition.
\begin{defin}
	A morphism between $V$-quantum B-algebras $f: (X, a)  \rightarrow (Y,b)$ is a $V$-functor that satisfies
	$$ f(x \rimp y) \leq f(x) \rimp f(y),  \mbox{  }  f(x \limp y) \leq f(x) \limp f(y), \mbox{ }  u_Y \leq f(u_X).$$
	If also
	$$ f(x) \rimp f(y) \leq f(x \rimp y),  \mbox{  }  f(x) \limp f(y) \leq f(x \limp y)),$$
	then $f$ is called strict.
\end{defin}
\begin{re}
We can easily derive these conditions from \ref{ProFun}. As an example we show how to derive  $u_Y \leq f(u_X)$.\par\medskip

	\AxiomC{$J(-) \leq U(f(-))$}
	\UnaryInfC{$(u_x)_{*} \leq f^{*}\cdot (u_Y)_{*}$ } (from $f_{*} \dashv f^*$)
	\UnaryInfC{$f_{*} \cdot (u_X)_{*} \leq (u_Y)_{*}$ }
	\UnaryInfC{$b(f(u_X),=) \leq b(u_Y, =)$ }
	\UnaryInfC{$u_Y \leq f(u_X)$ }
	\DisplayProof

\end{re}
The following proposition follows from Proposition \ref{Fully}.
\begin{prop}
	The category of $V$-quantum B-algebras is a full subcategory of $\Multi{V}$.
\end{prop}
\begin{re}
	Let $ (X, a) $ be a $V$-quantum B-algebra, if we unravel the definition contained in \ref{Fully}, we have that the corresponding $(L,V)$-structure is given as follows. For a pair $(\underline{x},y)$ define:
	$$\tilde{a}((x_1,...,x_n), y) = \tilde{a}((x_1,...,x_{n-1}),x_n \rimp y) = ... = a(x_1, x_2 \rimp  ... x_n \rimp y),$$
	$$\tilde{a}((-), y ) = a(u_X, y), \ \ \mbox{   where $(-)$ is the empty list}.$$
\end{re}
\begin{re}
	Let $Q$ be a $V$-quantale. We can view $Q$ as a $V$-quantum B-algebra and as an $(L,V)$-category. It is easy to see that the two structures (call them $Q$ and $Q_L$) coincide; indeed, we have
	$$Q((q_1,...,q_n),w) = Q(q_1, q_2 \rimp ... q_n \rimp w) = Q(q_1 \mult ... \mult q_n, w) = Q_L((q_1,...,q_n),w).$$

\end{re}

\subsection{Injectives and Injective Hulls}\label{InjHulls}
In Section \ref{Intermezzo} we recalled the notion of injective object and of injective hull with respect to a class of morphisms $J$. As we did for $\Vcat$, from now on, when we write 'injective' we always mean injective with respect to fully faithful functors. Since $\Multi{V}$ serves as a common roof under which we perform all the constructions, when we are dealing with promonoidal $V$-categories, $V$-quantum B-algebras, etc. 'injective' always refers to the corresponding notion in $\Multi{V}$.\par\medskip

In the context of $(T,V)$-categories, Hofmann proved in \cite{HOFMANN2011283} the following theorem.
\begin{teorema}\label{KZ}
	A separated $(L,V)$-category $(X,a)$ is injective iff it is cocomplete.
\end{teorema}
We can go further in our analysis and characterize them as $V$-quantales.
\begin{teorema}
	A separated $(L,V)$-category $(X,a)$ is injective with respect to fully faithful functors iff it is a $V$-quantale.
\end{teorema}
\begin{re}
	 Since we defined $V$-quantales to be monoids in the monoidal category $\Alg{V}$, which is equivalent to the category of separated cocomplete $V$-categories, we restrict our attention to separated $(L,V)$-categories. This also allows us to avoid pseudo-algebras.
\end{re}
Before we give a direct proof, we state a useful lemma.
\begin{lem}
	Let $Q$ be a $V$-quantale and $(Y,b)$ be an $(L,V)$-category. Let $f :Y \rightarrow Q$ be an $(L,V)$-functor. Then there exists an $(L,V)$-functor $g : \mathbb{D}_L(Y) \rightarrow Q$ that makes the following diagram commute
		\[
	\begin{tikzcd}
	Y \arrow[d, "\Yo{Y}", swap] \arrow[r, "f"] & Q \\
	\mathbb{D}_L(Y).  \arrow[ur, dotted, "g" , swap] &
	\end{tikzcd}
	\]
\end{lem}
\begin{proof}
	Let $g = \mathtt{Lan}_{\Yo{X}}(f)$. Since $\mathtt{Lan}_{\Yo{X}}(f) = \Tcolim{L}(y_{\circledast},f)(1)$, which. by Proposition \ref{LCol}, is the same as $\colim(e_1^{\circ} \cdot \hat{y_{\circledast}}, \alpha Lf )$ and $Q$ is a $V$-quantale, $g$ is well defined (as a function). Moreover it makes the diagram commute. \\
	If we prove that there exists an $(L,V)$-functor $R : Q \rightarrow \mathbb{D}_L(Y)$ such that $R^{\circledast} \simeq g_{\circledast}$, we can conclude that $g$ is an $(L,V)$-functor too.\\
	By unravelling the definition of colimit, we have that
	\begin{alignat*}{2}
	g(j) &= \bigvee_{\underline{y} \in LY} \mathbb{D}_L(Y)[L\Yo{X}(\underline{y}),j] \odot (f(y_1) \mult ... \mult f(y_n)) \\
	& \mbox{(by the Yoneda lemma)} \\
	& = \bigvee_{\underline{y} \in LY} j(\underline{y}) \odot (f(y_1) \mult ... \mult f(y_n)).
	\end{alignat*}
	Define $R(w) = Q(Lf(-), w)$. Since $R$ is the mate of $f_\circledast$, by Remark \ref{mate} it follows that $R$ is an $(L,V)$-functor. In order to conclude, we have to show that
	$$g_\circledast(\underline{j},w) = R^\circledast(\underline{j},w).$$
	Without loss of generality, we can assume that $\underline{j} = (j_1,j_2).$\\
	We compute:
	\begin{alignat*}{2}
	g_\circledast(\underline{j},w)& = Q(g(j_1) \mult g(j_2) ,w) \\
	& = Q(\bigvee_{\underline{y} \in LY} j_1(\underline{y}) \odot (f(y_1) \mult ... \mult f(y_n)) \mult \bigvee_{\underline{z} \in LY} j_2(\underline{z}) \odot (f(z_1) \mult ... \mult f(z_n)), w ) \\
	& \mbox{(since the multiplication of $Q$ preserves colimits)} \\
	& = \bigwedge_{\underline{y} \in LY} \bigwedge_{\underline{z} \in LY} [ j_1(\underline{y}) \otimes j_2(\underline{z}),Q((f(y_1) \mult ... \mult f(y_n)) \mult (f(z_1) \mult ... \mult f(z_n)), w)  ] \\
	&= \bigwedge_{(\underline{y}, \underline{z}) \in L^2Y}  [Lj(\underline{y}, \underline{z}), R(w)(\underline{y}; \underline{z})] \\
	& \mbox{(by the definition of the $(L,V)$-structure on $\mathbb{D}_L(Y)$, as explained in Remark \ref{DLstructure})} \\
	&= \mathbb{D}_L(Y)[\underline{j},R(w)]\\
	& = R^\circledast(\underline{j},w).
	\end{alignat*}
	Thus $g$ is an $(L,V)$-functor.\\
\end{proof}
\begin{proof}(Of the theorem).\\
	$\Leftarrow).$ Suppose that $Q$ is a $V$-quantale and let $g : (X,a) \rightarrow Q$ be an $(L,V)$-functor. Suppose that $f : (X,a) \rightarrow (Y,b)$ is a fully faithful $(L,V)$-functor. Consider the Yoneda embedding $\Yo{Y} : Y \rightarrow \mathbb{D}_L(Y)$. By the previous lemma, there exists an extension such that the following diagram commutes
	\[
	\begin{tikzcd}
	 X \arrow[r, "g"] \arrow[d,"f", swap] & Q \\
	 Y \arrow[d, "\Yo{Y}", swap] & \\
	 \mathbb{D}_L(Y). \arrow[uur, dotted, "" ] &
	\end{tikzcd}
	\]
	Hence $Q$ is injective.\par\medskip
$\Rightarrow).$ Suppose that $Q$ is injective. By Theorem \ref{KZ}, it follows that $Q$ is a cococomplete $(L,V)$-category, hence it is representable. Since the $2$-functor
	$$(-)_0 : \Multi{V} \rightarrow \Vcat$$
	sends injective objects in $\Multi{V}$ to injective objects in $\Vcat$ and since injectives in $\Vcat$ are cocomplete $V$-categories, it follows that $Q$ is a cocomplete $V$-category. It remains to show that the monoidal structure on $Q$ preserves $V$-colimits (in each variable).\\
	Since $Q$ is representable, $ \mathbb{D}(Q)$ becomes a $V$-quantale with respect to Day's convolution product (as explained in Remark \ref{Day}); moreover, the Yoneda embedding
	$$ \Yo{Q} : Q \rightarrow \mathbb{D}(Q)$$
	is a strong monoidal $V$-functor. Thus, we have
	\begin{alignat*}{2}
	Q(\underline{x},w) &= Q(x_1 \mult ... \mult x_n, w) \\
	&= \mathbb{D}(Q)(\Yo{Q}(x_1 \mult ... \mult x_n), \Yo{Q}(w)) \\
	&= \mathbb{D}(Q)(\Yo{Q}(x_1) \mult_D ... \mult_D \Yo{Q}(x_n), \Yo{Q}(w))  \\
	&=  \mathbb{D}(Q)(L\Yo{Q}(\underline{x}), \Yo{Q}(w)),
	\end{alignat*}
	therefore $\Yo{Q}$ is fully faithful also in $\Multi{V}$.\\
	Because $Q$ is injective, there exists an extension
\[
	\begin{tikzcd}
	Q \arrow[r,equal] \arrow[d,  "\Yo{Q}", swap] & Q \\
	\mathbb{D}(Q) \arrow[ur, dotted, "h", swap]
	\end{tikzcd}
\]
	with $h \Yo{Q} = \Id_Q$.\\
	Since the same holds in $\Vcat$, by Theorem \ref{Coco}, $h$ must be of the form
	\begin{alignat*}{2}
	h  :\mathbb{D}(Q) &\rightarrow Q. \\
		j        &\mapsto \colim(j,\Id)
	\end{alignat*}
	In order to show that $Q$ is a $V$-quantale, by Theorem \ref{Coco}, it is sufficient to show that the monoidal structure preserves colimits of the form $\colim(j,\Id) $, for $j \in \mathbb{D}(Q)$. Let $x \in Q$ and $j\in \mathbb{D}(Q)$. We have
	\begin{alignat*}{2}
	x \mult h(j) &= h(\Yo{Q}(x)) \mult h(j) \\
	& \mbox{ (since $h$ is lax monoidal) } \\
	& \leq h(\Yo{Q}(x) \mult_D j) \\
	& \mbox{ (since every presheaf is the colimit of the Yoneda embedding)} \\
	& = h( \Yo{Q}(x) \mult_D \colim(j, \Yo{Q})) \\
	& \mbox{ (since Day's convolution preserves colimits) } \\
	& = h( \colim(j, \Yo{Q}(x) \mult_D \Yo{Q}(=))) \\
	& \mbox{ (since the Yoneda embedding is strong monoidal) } \\
	& = h(\colim(j, \Yo{Q}(x \mult=))) \\
	& \mbox{ (since $h$ preserves colimits being left adjoint to $\Yo{Q} $ in $\Vcat)$ }\\
	& = \colim(j, h(\Yo{Q}(x \mult=))) \\
	& = \colim(j, x\mult (=)).
	\end{alignat*}
	Hence, because $\colim(j, x\mult (=)) \leq x \mult \colim(j,\Id)$ follows from the universal property of colimit, we have
	$$ x \mult \colim(j,\Id) = \colim(j, x\mult (=)).$$
	Thus $Q$ is a $V$-quantale.\\
\end{proof}
With the aid of the previous theorem, we can characterize those full subcategories of $\Multi{V}$ whose injectives are $V$-quantales in the following way.
\begin{teorema}\label{Inj}
	Let $C$ be a full subcategory of $\Multi{V}$ which contains all $V$-quantales, and let $(X,a)$ be a separated object in $C$. Then $(X,a)$ is injective iff it is a $V$-quantale.
\end{teorema}
\begin{proof}
	$\Leftarrow).$ $V$-quantales are injectives in $\Multi{V}$ and $C$ is a full subcategory of it. Thus, if $X$ is a $V$-quantale, then it is injective in $C$ too.\par\medskip
$\Rightarrow).$ If $X$ is an injective object in $C$, then we have the following diagram in $\Multi{V}$:
\[
	\begin{tikzcd}
	X \arrow[r, equal]  \arrow[d,  "\Yo{X}", swap]  &  X\\
	  \mathbb{D}_L(X). \arrow[ur, dashed, "\exists h", swap]
	\end{tikzcd}
\]
	The commutativity of the diagram implies that $h  \Yo{X} = \Id.$ This means that $ h :  \mathbb{D}_L(X) \rightarrow X$ defines an algebra structure for the monad $ \mathbb{D}_L(-)$, but because algebras for this monads are $V$-quantales, this implies that $X$ is a $V$-quantale too.
	\end{proof}
As a corollary of the previous theorem we get a characterization of injectives as $V$-quantales in each of the categories displayed in the following diagram.
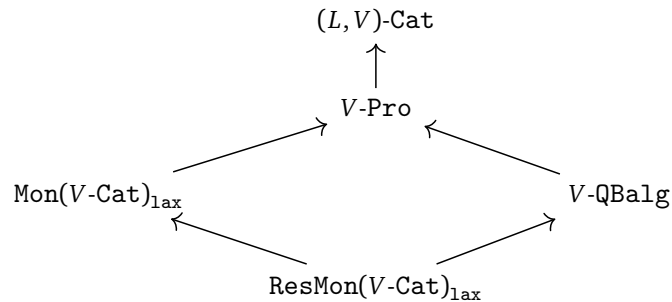
\[
\begin{tikzcd}
& \Multi{V} &  \\
& \Pro \arrow[u] & \\
\mathtt{Mon}(V \mbox{-}\mathtt{Cat})_{\mathtt{lax}}  \arrow[ur] & &V \mbox{-}\mathtt{QBalg} \arrow[ul] \\
& \mathtt{ResMon}(V \mbox{-}\mathtt{Cat})_{\mathtt{lax}} \arrow[ur] \arrow[ul]
\end{tikzcd}
\]
\begin{re}\label{Discussion}
	Notice that, when $V = \mathbf{2}$, the previous diagram and Theorem \ref{Inj} allow us to recover both Theorem $4.1$ of \cite{InjHulls} and Proposition $4$ of \cite{2016THECO}.
\end{re}
In order to build injective hulls we mimic what we did for $\Vcat$ and introduce the notion of \textit{dense morphism}. The difference here is that, due to the nature colimits in our "roof" category have, in our definition we consider $\Tcolim{L}$ instead of $\colim$.
\begin{defin}
	Let $(X,a)$ be in $\Multi{V}$. A fully faithful $(L,V)$-functor $i : (X,a )\rightarrow Q$, where $Q$ is a $V$-quantale (seen as an object of $\Multi{V}$) is called \textit{dense}, if we can write every $q \in Q$ as
	$$ q = \limit
	( i^{*} \cdot q_{*}, i), \ \ q =  \Tcolim{L}( i_{\circledast} \circ q^{\circledast},i)(1).$$
	If only the latter holds, we call $f$ \textit{pre-dense}.
\end{defin}
\begin{re}
	Notice that the limit in the previous definition is a $V$-enriched limit. We consider the $2$-functor
	$$ (-)_0 : \Multi{V} \rightarrow \Vcat,$$
	and take the limit accordingly.
\end{re}
\begin{re}\label{Discussion2}
	Using the decomposition of limits (see Remark \ref{Deco1}), we can write the first condition as
	$$\limit( i^{*} \cdot q_{*}, i) \simeq  \bigwedge_{x \in X} Q(q,i(x)) \pitchfork i(x)$$
	which in the ordered case it reduces to the more familiar formula
	$$ q =  \bigwedge_{\{x \mbox{ : } q \leq i(x)\}} i(x).$$
	If we unravel the second condition and use Corollary \ref{Deco}, we have
	\begin{alignat*}{2}
	\Tcolim{L}( i_{\circledast} \circ q^{\circledast},i)(1)   & \simeq  \bigvee_{\underline{x} \in LX} Q(\alpha Li(\underline{x}),q) \odot \alpha Li(\underline{x}) \\
	& = \bigvee_{\underline{x} \in LX} Q(i(x_1) \mult ... \mult i(x_n),q) \odot (i(x_1) \mult ... \mult i(x_n))
	\end{alignat*}
	which in the order case it reduces to
	$$q = \bigvee_{\{\underline{x} \mbox{ : } i(x_1) \mult ... \mult i(x_n)\leq q\} } i(x_1) \mult ... \mult i(x_n).$$
	The last condition shows how, in the ordered case, $X$ generates the quantale $Q$. This also shows how, in the enriched case, $(L,V)$-categories (more specifically the theory of $(L,V)$-colimits) are able to "capture", in an elegant and concise way, what we might call "monoidal colimits". That is to say, we provide a way to capture the idea of "$V$-quantale generated by a $V$-category", an intuitive construction that would be difficult to formalize without $(L,V)$-categories.
\end{re}
\begin{defin}
	A morphism $i : (X,a )\rightarrow Q$, where $Q$ is a $V$-quantale and $X$ is a $V$-quantum B-algebra, is dense if it is dense in $\Multi{V}.$
\end{defin}
\begin{re}
	We stress the fact that we are considering embeddings of $(L,V)$-categories into $V$-quantales which are cocomplete $(L,V)$-categories and complete $V$-categories.
\end{re}
\begin{lem}
	Let $i :  (X,a) \rightarrow (Y,b)$ be a fully faithful $V$-functor between $V$-categories, with $(Y,b)$ complete. Then
	$$ T  : (Y,b) \rightarrow (Y,b), \ \ y \mapsto \limit(i^{*} \cdot y_{\star}, i),$$
	is a $V$-functor.
\end{lem}
\begin{proof}
	Notice that $T$ can be written as the composite of the following $V$-functors
	\[
	\xymatrix{
		Y \ar[r]^{ \CoYo{Y}}  & \mathbb{U}(Y) \ar[r]^{i^{-1}} & \mathbb{U}(X)  \ar[r]^{ \mathbb{U}(i)} & \mathbb{U}(Y) \ar[r]^{ \limit} & Y,
	}
	\]
	where:
	\begin{itemize}
		\item $\CoYo{Y} : y \mapsto y_{*} = b(y, =);$
		\item $ i^{-1} : \psi \mapsto i^{*} \cdot \psi;$
		\item $ \mathbb{U}(i): \phi \mapsto i_{*} \cdot \phi;$
		\item $ \limit : \gamma \mapsto \limit(\gamma, \Id).$
	\end{itemize}

\end{proof}
\begin{lem}
	Let $i :  (X,a) \rightarrow (Y,b)$ be a fully faithful $V$-functor between $V$-categories, with $(Y,b)$ complete. Then $T \cdot i \simeq i$, meaning that, for all $x \in (X,a)$, $T(i(x)) \simeq \nobreak i(x)$.
\end{lem}
\begin{proof}
	The result follows when we notice that, since $i$ is fully faithful,
	$$ i^{*}\cdot i(x)_{*} = x_{*},$$
	and thus that $x_{*}\cdot i(x)^{*} = x_{*} \cdot x^{*} \cdot i^{*} \leq i^{*},$ in a universal way. This proves our desired result.
	\\
\end{proof}
\begin{prop}\label{Lax}
	Let $ i : (X,a )\rightarrow Q$, be a pre-dense strict morphism in $V$-$\mathtt{QBalg}$, where $Q$ is a $V$-quantale. \\Then $T : Q \rightarrow Q$ defined above is a lax monoidal monad.
\end{prop}
\begin{proof}
	We want to prove that, for all $q,w\in Q$:
	\begin{itemize}
		\item $q \leq T(q);$
		\item  $k_Q \leq T(k_Q);$
		\item $T^2(q) = T(q);$
		\item $T(q) \mult T(w) \leq T(q \mult w).$
	\end{itemize}
	We have the following chain of adjunctions
	\begin{center}
	\begin{tikzcd}
	 \mathbb{U}(X) \ar[r,bend left,"\mathbb{U}(i)",""{name=A, below}] &   \mathbb{U}(Q)\ar[l,bend left," i^{-1} ", ""{name=B,above}] \ar[from=B, to=A, symbol=\dashv] \ar[r,bend left,"\limit",""{name=C, below}]  & Q \ar[l,bend left,"(-)_{\star} = \CoYo{Q} ",""{name=D,above}] \ar[from=D, to=C, symbol=\dashv]
	\end{tikzcd}
	\end{center}
	where $ i^{-1} : \psi \mapsto i^{*} \cdot \psi.$\\
	Since $T(q) \simeq \limit(i^{*} \cdot q_{*}, i),$ it follows that
	$$k \leq Q(z,T(q)) = \mathbb{U}(X)(i^{*}\cdot z_{*}, i^{*}\cdot q_{*})  \ \ \mbox{ iff, for all $x \in X,$ } Q(q,i(x)) \leq Q(z,i(x)).$$
	From this it follows immediately that $q \leq T(q)$, hence also $k_Q \leq T(k_Q).$ \\
	From the latter, since $T$ is a $V$-functor, it follows that $T(q) \leq T^2(q)$. For the other direction, we need to prove that
	$$ Q (q, i(x)) \leq Q(T^2(q), i(x)).$$
	Since $T$ is a $V$-functor, it follows from the previous lemma that
	$$  Q (q, i(x)) \leq Q(T(q) , T(i(x))) = Q(T(q), i(x)) \leq Q(T^2(q), i(x)).$$
	Let's prove that $T(q) \mult T(w) \leq T(q \mult w).$
	Let $x \in X$, by using the pre-denseness of $i$, we have
	$$Q(q \mult w, i(x)) = Q(q \mult \Tcolim{L}(\phi, \Id)(1), i(x)),$$
	where, for the sake of notation, $\phi = i_{\circledast}\circ b^{\circledast}$.\\
	By unravelling the definition of colimit, we get
	\begin{alignat*}{2}
	Q(q \mult w, i(x)) & = Q(q \mult \Tcolim{L}(\phi, i(y))(1), i(x)) \\
	& \mbox{(from Proposition \ref{LCol})} \\
	& = Q(q \mult \colim(e_1^{\circ} \cdot \phi, \alpha Li(\underline{y})), i(x)) \\
	&= \limit(e_1^{\circ} \cdot \phi,  Q(q \mult \alpha Li(\underline{y}), i(x)))  \\
	& = \limit(e_1^{\circ} \cdot \phi,  Q(q \mult (i(y_1) \mult ... \mult i(y_n)), i(x)))  \\
	& \mbox{(by inductively applying adjunction)} \\
	& = \limit(e_1^{\circ} \cdot \phi,  Q(q, i(y_1) \rimp ( ... \rimp i(y_n)\rimp i(x))) ) \\
	& (\mbox{by inductively applying $i(x \rimp y) = i(x) \rimp i(y)$})\\
	& = \limit(e_1^{\circ} \cdot \phi,  Q(q, i(y_1 \rimp ...\rimp (y_n\rimp x)) )) \\
	& \mbox{(being $T$ a $V$-functor and from the previous lemma)}\\
	& \leq \limit(e_1^{\circ} \cdot \phi,  Q(T(q), i(y_1 \rimp ...\rimp(y_n\rimp x))))\\
	& \mbox{(by inductively applying adjunction, plus $i(x \rimp y) = i(x) \rimp i(y)$ backward)} \\
	&\limit(e_1^{\circ} \cdot \phi, Q(T(q) \mult (i(y_1) \mult ... \mult i(y_n)), i(x)) ) \\
	& = Q(T(q) \mult w, i(x)).
	\end{alignat*}
	By repeating the same argument, but now expressing $T(q)$ as a weighted colimit, and by using $\limp$ instead of $\rimp$, we get
	$$Q(q \mult w, i(x)) \leq Q(T(q) \mult w, i(x)) \leq Q(T(q) \mult T(w), i(x)).$$
\end{proof}
\begin{cor}
	Let $ i : (X,a )\rightarrow Q$, be a pre-dense strict morphism in $V$-$\mathtt{QBalg}$, where $Q$ is a $V$-quantale. Then $i' : X \rightarrow Q^T$ is dense.
\end{cor}
\begin{proof}
Since $T$ is a left adjoint $V$-functor and a morphism of $V$-quantales, it follows that $T$ is also a left adjoint $(L,V)$-functor, since
 $$Q^T(T(q_1) \mult' ... \mult' T(q_n), w) = Q^T(T(q_1 \mult ... \mult q_n), w ) = Q(q_1 \mult ... \mult q_n, g(w)).$$
 This implies that $T$ preserves $(L,V)$-colimits as well, hence that $i'$ is pre-dense.\\
 Every element of $Q^T$ satisfies $T(q) \simeq q,$ hence it follows that $i'$ is dense.\\
\end{proof}
\begin{prop} \label{Strict}
	Every $V$-quantum B-algebra $X$ admits a strict morphism into a $V$-quantale.
\end{prop}
\begin{proof}
	Let $(X, a)$ be the corresponding object in $\Multi{V} $. Consider the Yoneda embedding in $\Multi{V}$
	$$ \Yo{X} : X \rightarrow \mathbb{D}_L(X),  \ \ x \mapsto a(-, x).$$
	We have that
	\begin{alignat*}{2}
	a((\underline{x}, w), z) & = \mathbb{D}_L(X)[L\Yo{X}((\underline{x};w)), \Yo{X}(z)] \\
	& = \mathbb{D}_L(X)[(\alpha L\Yo{X}(\underline{x})) \mult \Yo{X}(w), \Yo{X}(z)] \\
	& = \mathbb{D}_L(X)[\alpha L\Yo{X}(\underline{x})), \Yo{X}(w) \rimp \Yo{X}(z)) ] \\
	& =\mathbb{D}_L(X)[ L\Yo{X}(\underline{x})), \Yo{X}(w) \rimp \Yo{X}(z))].
	\end{alignat*}
	Since, by definition, we have that
	$$ a((\underline{x}; w), z)  =  a(\underline{x}, w \rimp z) = \mathbb{D}_L(X)[L\Yo{X}(\underline{x}), \Yo{X}(w \rimp  z)]; $$
	from which it follows that, for all $\underline{x} \in LX$,
	$$ \mathbb{D}_L(X)[ L\Yo{X}(\underline{x})), \Yo{X}(w) \rimp \Yo{X}(z)) ] =\mathbb{D}_L(X)[L\Yo{X}(\underline{x}), \Yo{X}(w \rimp  z) ]; $$
	hence, by representability,
	$$  \Yo{X}(x_2) \rimp \Yo{X}(x) =  \Yo{X}(x_2 \rimp  x).$$
	By using the same argument, but starting from $ a ( (w; \underline{x}), z),$
	we can also prove that
	$$  \Yo{X}(x_2) \limp \Yo{X}(x) =  \Yo{X}(x_2 \limp  x).$$
\end{proof}
\begin{re}
	Notice that $\mathbb{D}_L(X)$ is an object of $V$-$\mathtt{QBalg}$ and, when viewed as an $(L,V)$-category, its $(L,V)$-structure $\tilde{Q}_{\mathbb{D}_L(X)}$ is the same as $\alpha \cdot Q_{\mathbb{D}_L(X)}$ by adjointness.
\end{re}
\begin{cor}
Let $X$ be a $V$-quantum B-algebra. Then $\Yo{X} : X \rightarrow \mathbb{D}_L(X)$ is a pre-dense strict morphism.
\end{cor}
\begin{proof}
Every element of $\mathbb{D}_L(X)$ is a weighted colimit of $\Yo{X}$ in a canonical way. From the Yoneda lemma it follows that
$$ \phi \simeq \Tcolim{L}( \Yo{X}, \phi)(1) \simeq \Tcolim{L}( \Yo{X}, (\Yo{X})_{\circledast} \circ \phi^{\circledast})(1) .$$
\end{proof}
\begin{cor}
	For every $V$-quantum B-algebra $X$, there exists a $V$-quantale $Q$ and a dense morphism $i : X \rightarrow Q$.
\end{cor}
\begin{proof}
	Apply Proposition \ref{Lax} to the Yoneda embedding
	$$ \Yo{X} : X \rightarrow \mathbb{D}_L(X).$$
\end{proof}
\begin{re}
	As we already remarked before, our class $J$ is the class of fully faithful functors in $\Multi{V}$.
\end{re}
\begin{prop}
	\label{Dense}
	Every dense morphism in $\Multi{V}$ is essential.
\end{prop}
\begin{proof}
	Suppose that $i:(X,a) \rightarrow Q$ is dense and let $g : Q \rightarrow (Z,b)$ be such that $\restr{g}{X}$ is fully faithful.\\
	We want to prove that
	$$ b(Lg(\underline{y}),g(z)) \leq Q(y_1 \mult y_2, z), \mbox{ (where, for simplicity } \underline{y} = (y_1, y_2)).$$
	Since $i$ is dense, we can write
	$$y_1  = \Tcolim{L}(i_{\circledast} \circ y_1^{\circledast} , i )(1), \ \ y_2 =  \Tcolim{L}(i_{\circledast} \circ y_2^{\circledast} , i )(1), \ \ z = \limit(i^{*} \cdot z_*, i) . $$
	In order to increase readability, we use the following notation:
	$$j_1 = i_{\circledast} \circ y_1^{\circledast}, \ \ j_2 = i_{\circledast} \circ y_2^{\circledast}, \ \ j_3 = i^{*} \cdot z_*.$$
	We have that
	\begin{alignat*}{1}
	Q(y_1 \mult y_2, a) &= Q(\colim(e^{\circ}_1 \cdot \hat{j_1},\alpha Li(\underline{x})) \mult \colim(e^{\circ}_1 \cdot \hat{j_1},\alpha Li(\underline{w})), \limit(j_3,i(x) ))\\
	&\mbox{(since $\mult$ preserves colimits, being Q a $V$-quantale)} \\
	& = \limit( e^{\circ}_1 \cdot \hat{j_1},  \limit(e^{\circ}_1 \cdot \hat{j_2}, \limit(j_3,
	Q(\alpha Li(\underline{x}) \mult \alpha Li(\underline{w}) , i(x))))) \\
	& \mbox{(being $i$ fully faithful)} \\
	& = \limit( e^{\circ}_1 \cdot \hat{j_1},  \limit(e^{\circ}_1 \cdot \hat{j_2}, \limit(j_3,
	a((\underline{x} ;\underline{w}) , x)))) \\
	& \mbox{(since $\restr{g}{X}$ is fully faithful)} \\
	& = \limit( e^{\circ}_1 \cdot \hat{j_1},  \limit(e^{\circ}_1 \cdot \hat{j_2}, \limit(j_3,b(Lg(\underline{x} ;\underline{w}) , g(x))
	))).
	\end{alignat*}
	By using the Yoneda embedding $\Yo{Z} : (Z,b) \rightarrow \mathbb{D}_L(Z)$, we have that $\Yo{Z}  g$ is fully faithful (restricted to $X$, of course). Moreover, we have (define $f = \Yo{Z} g i $)
	$$ \Yo{Z} g(\limit(j_3,i(x) )) \leq \limit(j_3, f(x)), \ \ (1)$$
	and
	$$\colim(e^{\circ}_1 \cdot \hat{j_1},\alpha Lf(\underline{x})) \leq \Yo{Z}  g(\colim(e^{\circ}_1 \cdot \hat{j_1},\alpha Li(\underline{x}))), \ \ \colim(e^{\circ}_1 \cdot \hat{j_2},\alpha Lf(\underline{w})) \leq \Yo{Z}  g(\colim(e^{\circ}_1 \cdot \hat{j_2},\alpha Li(\underline{w})))$$
	which, when combined together, they give
	$$\colim(e^{\circ}_1 \cdot \hat{j_1},\alpha Lf(\underline{x})) \mult \colim(e^{\circ}_1 \cdot \hat{j_2},\alpha Lf(\underline{w}))\leq \Yo{Z}  g(\colim(e^{\circ}_1 \cdot \hat{j_1},\alpha Li(\underline{x}))) \mult \Yo{Z}  g(\colim(e^{\circ}_1 \cdot \hat{j_2},\alpha Li(\underline{w}))). \ \ (2)$$
$$
	\begin{aligned}&\texttt {colim}(e^{\circ }_1 \cdot \hat{j_1},\alpha Lf({\underline{x}})) *\texttt {colim}(e^{\circ }_1 \cdot \hat{j_2},\alpha Lf({\underline{w}})) \\&\quad \le {\mathbf {y}}_{Z} g(\texttt {colim}(e^{\circ }_1  \cdot \hat{j_1},\alpha Li({\underline{x}}))) *{\mathbf {y}}_{Z} g(\texttt {colim}(e^{\circ }_1 \cdot \hat{j_2},\alpha Li({\underline{w}}))). (2) \end{aligned}
		$$
	Thus
	\begin{alignat*}{2}
	Q(y_1 \mult y_2, z)  &= \limit( e^{\circ}_1 \cdot \hat{j_1},  \limit(e^{\circ}_1 \cdot \hat{j_2}, \limit(j_3,b(Lg(\underline{x} ;\underline{w}) , g(x_k))
	))) \\
	& \mbox{(by the Yoneda lemma)} \\
	& = \limit( e^{\circ}_1 \cdot \hat{j_1},  \limit(e^{\circ}_1 \cdot \hat{j_2}, \limit(j_3,
	\mathbb{D}_L(Z)[ (\alpha Lf(\underline{x};\underline{w}), f(x)])))  \\
	& = \limit( e^{\circ}_1 \cdot \hat{j_1},  \limit(e^{\circ}_1 \cdot \hat{j_2}, \limit(j_3,
	\mathbb{D}_L(Z)[ (\alpha Lf(\underline{x})) \mult (\alpha Lf(\underline{w})), f(x)])))  \\
	&\mbox{(since $\mult$ preserves colimits, being $\mathbb{D}_L(Z)$ a $V$-quantale)} \\
	& = \mathbb{D}_L(Z) [ \colim(e^{\circ}_1 \cdot \hat{j_1},\alpha Lf(\underline{x})) \mult \colim(e^{\circ}_1 \cdot \hat{j_2},\alpha Lf(\underline{w})),\limit(j_3, f(x))] \\
	& \mbox{(using (1) + (2))} \\
	& \geq \mathbb{D}_L(Z)[ \Yo{Z} g(y_1) \mult \Yo{Z} g(y_2), \Yo{Z} g(z)] \\
	& = \mathbb{D}_L(Z)[L(\Yo{Z} g)(\underline{y}), \Yo{Z} g(z)] \\
	& \mbox{(by the Yoneda lemma)} \\
	&= b(Lg(\underline{y}),g(z)).
	\end{alignat*}
	This proves that
	$$ b(Lg(\underline{y}),g(z)) \leq Q(\underline{y}, z).$$
\end{proof}
\begin{re}
	Since the notion of denseness in $V$-$\mathtt{QBalg}$ and $\Multi{V}$ coincide, the same result holds for dense morphisms in $V$-$\mathtt{QBalg}$.
\end{re}
\begin{teorema}
A morphism $ i :(X,a) \rightarrow Q$ in $V$-$\mathtt{QBalg}$ is dense iff it is essential.
\end{teorema}
\begin{proof}
	The first part follows from the previous proposition. To show the converse implication, suppose that $ i :(X,a) \rightarrow Q$ is essential. Consider the left Kan extension of $i$ along the Yoneda embedding and take its image factorization, the result is the $V$-quantale generated by $X$ in $Q$:
	\begin{center}

	\begin{tikzcd}
	 X \arrow[r,  "\Yo{X}"] \arrow[dr, "j", swap] & \mathbb{D}_L(X) \arrow[r, "\mathtt{Lan}_{y}(i)"] \arrow[d, two heads] & Q \\
	   									& \mathtt{Im}(\mathtt{Lan}_{y}(i)). \arrow[ur]
	\end{tikzcd}
		\end{center}
	Because $\Yo{X}$ is a strict pre-dense morphism, it follows that $j$ is a strict pre-dense morphism too. Let $T$ be the monad induced by $j$ as in Proposition \ref{Lax}. Consider the following diagram
	\begin{center}
		\begin{tikzcd}
	X \arrow[r,  "j"] \arrow[dr,"j'", swap] & \mathtt{Im}(\mathtt{Lan}_{y}(i)) \arrow[r] \arrow[d,"\pi_T" two heads] & Q \arrow[dl, "f"]\\
	& (\mathtt{Im}(\mathtt{Lan}_{y}(i)))^T
	\end{tikzcd}
		\end{center}
	where $f$ is obtained by applying the injective property to $j'$. In particular, since $i$ essential, it follows that $f$ is fully faithful. Since $\pi_T$ is full and essentially surjective, it follows that $f$ is full and essentially surjective too which implies that $f$ is an equivalence. Hence $Q \simeq (\mathtt{Im}(\mathtt{Lan}_{y}(i)))^T$ which implies that $i$ is dense.\\
	\end{proof}
\section{Injective Hulls of (L,V)-Categories}
In this section we apply the results we proved in the previous section to $\Multi{V}$. The idea is to use the Yoneda embedding to view any $(L,V)$-category as a subset of the $V$-quantale $\mathbb{D}_L(X)$ and then consider the $V$-quantum B-algebra it generated.\par\medskip

Due to the equational definition enriched quantum B-algebras admit, if we have a subset $X \subseteq Q$ of a $V$-quantale $Q$ that contains the unit $u_Q$, the $V$-quantum B-algebra generated by $X$ (inside $Q$) is obtained by inductively adding all the "missing" implications until the process reaches a "saturation". If the unit is not included in $X$, we have to add it in the first step of the construction; this might cause a problem, since we are "artificially" adding an element which is "alien" and can not be constructed inductively starting from elements of $X$.\par\medskip

In order to overcome this problem, we introduce a slight variation of $V$-quantum B-algebras: \textit{$V$-prequantum B-algebras}. A $V$-prequantum B-algebra is a promonoidal category $(X,P,J)$ where only $P$ is representable. In this way we obtain a new subcategory of $\Multi{V}$ where all the constructions made in the last section still work, since all it requires for them to work is the presence of an "implicational structure". Moreover, in the process of building the $V$-quantum B-algebra generated by a subset, we won't have to "artificially" add the unit of the $V$-quantale as we should have if we had considered \textit{vanilla} $V$-quantum B-algebras.
\begin{defin}
	A \textit{$V$-prequantum B-algebra} $(X,a, J)$ is a \textit{pre-representable} promonoidal $V$-category. This means it is a promonoidal $V$-category $(X,a,P,J)$ equipped with two binary operations $\rimp, \limp : X \times X \rightarrow X$ such that:
	\begin{itemize}
		\item $P(x,-,y) \simeq a(-, x\limp y);$
		\item $P(-,x,y) \simeq a(-, x\rimp y).$
	\end{itemize}
\end{defin}
\begin{defin}
	A morphism between $V$-prequantum B-algebras $f: (X, a, J)  \rightarrow (Y,b,U)$ is a $V$-functor that satisfies
	$$ f(x \rimp y) \leq f(x) \rimp f(y),  \mbox{  }  f(x \limp y) \leq f(x) \limp f(y), \mbox{ }  J(x) \leq U(f(x)).$$
	If also
	$$ f(x) \rimp f(y) \leq f(x \rimp y),  \mbox{  }  f(x) \limp f(y) \leq f(x \limp y)),$$
	then $f$ is called strict.
\end{defin}
As for \textit{vanilla} $V$-quantum B-algebras, we have:
\begin{prop}
	The category $V$-$\mathtt{PQBalg}$ of $V$-prequantum B-algebras is a full subcategory of $\Multi{V}$.
\end{prop}
\begin{re}
	Notice that, since the strictness condition involves only implications, the key results contained in Proposition \ref{Lax} and Proposition \ref{Strict} remain true when applied to $V$-prequantum B-algebras. Thus we can construct the injective hull of a $V$-prequantum B-algebra $(X,a, J)$ as done in the previous section.
\end{re}
The last step is to describe the $V$-prequantum B-algebra generated by a subset of a $V$-quantale $Q$.
	\begin{defin}
	Let $X \subseteq Q$ be a subset of a $V$-quantale $(Q,\mult, u_Q)$. The $V$-prequantum B-algebra generated by $X$, denoted by $X$, is the smallest $V$-prequantum B-algebra that contains $X$.
\end{defin}
\begin{prop} \label{Gen}
	Let $X \subseteq Q$ be a subset of a $V$-quantale $(Q,\mult, u_Q)$. Then
	$X = (\bigcup_i X_i, \tilde{Q}, J)$ is the $V$-prequantum B-algebra generated by $X$, where:
	\begin{itemize}
		\item $X_0 = X, \ \  X_{i+1} = \{x \rimp y, \ \ z\limp w, \ \ x,y,z,w \in X_i  \} \cup X_i;$
		\item $\tilde{Q}$ is the restriction of the $V$-structure on $Q$ to $\bigcup_i X_i$;
		\item $J(x)= Q(u_Q,x).$
	\end{itemize}
	\end{prop}
\begin{proof}
	It is straightforward to verify that $\bigcup_i X_i$ is a $V$-prequantum B-algebra; hence that $X \subseteq \bigcup_i X_i$.\\
	For the other inclusion it is sufficient to prove by induction that $X_i \subseteq X$. \\
\end{proof}
\begin{re}
	In the context of Proposition \ref{Gen}, note that $X \rightarrow Q$ is a strict embedding of $V$-quantum B-algebras.
\end{re}
Let $(X,a)$ be in $\Multi{V}$. Consider the Yoneda embedding $\Yo{X} : (X,a) \rightarrow \mathbb{D}_L(X)$. Define $X_{\mathbb{D}_L(X)}$ to be the $V$-prequantum B-algebra generated by $\Yo{X}(X)$ in $\mathbb{D}_L(X)$. By definition, the Yoneda embedding factorizes as
$$ X \rightarrow X_{\mathbb{D}_L(X)} \rightarrow \mathbb{D}_L(X).$$
Let $\tilde{X}$ be an injective hull for $X_{\mathbb{D}_L(X)}$ in $V$-$\mathtt{PQBalg}$. From Proposition \ref{Dense} it follows that $X_{\mathbb{D}_L(X)} \rightarrow \tilde{X}$ is essential (and also dense) in $\Multi{V}$.
Consider the composite
$$(X,a) \rightarrow X_{\mathbb{D}_L(X)} \rightarrow \tilde{X}.$$
\begin{prop}
	The composite $(X,a) \rightarrow X_{\mathbb{D}_L(X)} \rightarrow \tilde{X}$ is essential in $\Multi{V}.$
\end{prop}
\begin{proof}
	Let $f : \tilde{X} \rightarrow (Z,b)$ be such that $\restr{f}{X}$ is fully faithful. By composing with $\Yo{Z}$ we can suppose that $Z$ is a $V$-quantale, and thus that $f$ is a morphism in $(\Vcat^L)_{\mathtt{lax}}$.\\
	If we show that $\restr{f}{X}$ is fully faithful, since $\tilde{X}$ the injective hull of $X$, then it will follow that $f$ is fully faithful too, hence the result.\\
	First we are going to prove that
	$$  Z(Lf(\underline{y}), f(w)) \leq \tilde{a}_{X} ( (y_1, ..., y_n), w), \mbox{ with } y_1,...,y_n \in X \mbox{ and } w = \Yo{X}(x), \mbox{ with } x\in X.$$
	With a calculation similar to the one we did in Proposition \ref{Dense}, by using the properties of $f$ and the fact that $\restr{f}{X}$ is fully faithful, we have (where we consider $\underline{y} = (y_1,y_2)$ as a matter of convenience)
	\begin{alignat*}{2}
	\tilde{a}_{X} ( (y_1, y_2), w ) & =  a_{X} (y_1, y_2 \rimp w) \\
	& \mbox{ (because $X_{\mathbb{D}_L(X)}$ is the $V$-prequantum B-algebra generated by $\Yo{X}(X)$ in $\mathbb{D}_L(X)$)} \\
	& = \mathbb{D}_L(X)[y_1, y_2 \rimp \Yo{X}(x)] \\
	& = \mathbb{D}_L(X)[y_1 \mult y_2, \Yo{X}(x)] \\
	&= \limit(e_1^{\circ} \cdot j_1, \limit(e_1^{\circ} \cdot j_1, a((\underline{x}_1;\underline{x}_2) , x ))) \\
	& = \limit(e_1^{\circ} \cdot j_1, \limit(e_1^{\circ} \cdot j_1, Z (Lf(\underline{x}_1; \underline{x}_2)), f(w)))\\
	&  \geq Z(f(y_1) \mult f(y_2), f(w)).
	\end{alignat*}
	Hence
	$$ Z(Lf(\underline{a}), f(w)) \leq \tilde{a}_{X} ( \underline{y}, w).$$
	If $w\notin X$, by induction, we can suppose it is of the form $w = q \rimp z$ (or $q \limp z$), for $q,z \in \Yo{X}(X).$\\
	From $f(w) \mult f(q) \leq f(w \mult q) $ it follows that $f(w) \leq f(q) \rimp f(z)$; hence that
	\begin{alignat*}{2}
	Z(Lf(\underline{y}), f(w))  & \leq Z(Lf(\underline{y}), f(q) \rimp f(z) ) \\
	& = Z(f(y_1) \mult f(y_2), f(q) \rimp f(z) \\
	&= Z(f(y_1) \mult f(y_2) \mult f(q), f(z)) \\
	& = Z(Lf(\underline{y};q),  f(z) ).
	\end{alignat*}
	Since we supposed $z \in \Yo{X}(X)$, by the case we have already analyzed, it follows that
	$$ Z(Lf(\underline{y};q),  f(z) ) \leq \tilde{a}_{X} (\underline{y};q , z ) = \tilde{a}_{X} ( \underline{y}, q \rimp z ) ,$$
	hence that
	$$Z(Lf(\underline{y}), f(w)) \leq \tilde{a}_{X} ( \underline{y}, w),$$
	for every $w\in X$.\\
\end{proof}
\begin{re}
	Now it should be clear why we have introduced $V$-prequantum B-algebras. If we had considered \textit{vanilla} $V$-quantum B-algebras, in the last proposition we would have faced a problem, since we would not have been able to prove that
	$$Z(f(y_1) \mult f(y_2), f(u)) \leq \tilde{a}_{X} ( (y_1, y_2), u),$$
	where $u$ is the unit of $\mathbb{D}_L(X)$.
\end{re}
\begin{re}
	One might ask why we did not consider $V$-prequantum B-algebras in the first place instead of introducing \textit{vanilla} $V$-quantum B-algebras. The reason for our choice is due to the fact that $V$-prequantum B-algebras do not look as "natural" as $V$-quantum B-algebras do; requiring the representability of only the "multiplicative" part of a promonoidal category is an artifice we employed in order to overcome the minor problem we would have with the unit in the previous proposition.
\end{re}
With the aid of this proposition, we can prove our main theorem.
\begin{teorema}\label{Hulls}
	Let $C$ be a full subcategory of $\Multi{V}$ which contains all $V$-quantales. Then every object $(X,a)$ of $C$ admits an injective hull.
\end{teorema}
\begin{proof}
	From Theorem \ref{Inj} we get that the injectives in $C$ are $V$-quantales. From the previous theorem, for every object of $C$ there exists an essential embedding of it into a $V$-quantale. Hence the result. \\
\end{proof}
In particular this theorem applies to all the categories displayed in the following diagram
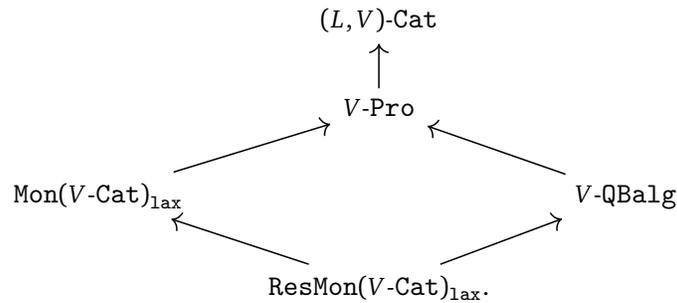
\[
\begin{tikzcd}
 & \Multi{V} &  \\
 & \Pro \arrow[u] & \\
 \mathtt{Mon}(V \mbox{-}\mathtt{Cat})_{\mathtt{lax}}  \arrow[ur] & & V \mbox{-}\mathtt{QBalg} \arrow[ul] \\
 & \mathtt{ResMon}(V \mbox{-}\mathtt{Cat})_{\mathtt{lax}}. \arrow[ur] \arrow[ul]
\end{tikzcd}
\]
\begin{re}
	Similarly to what we described in Remark \ref{Discussion}, when $V= \mathbf{2}$, Theorem \ref{Hulls} allows us to recover Theorem $5.8$ of \cite{InjHulls} and Theorem $1$ of \cite{2016THECO}. We stress again the important role played by ordered multicategories (and of $(L,V)$-categories in the general case) as the common roof in which all categorical constructions are made. Indeed, the crucial step is not only to recognize that both ordered monoids and ordered quantum B-algebras are examples of promonoidal ordered sets, but, mainly, it is to recognize that promonoidal ordered sets are examples of ordered multicategories. This is because ordered multicategories provide the right notion of colimit we need in order to deal with both the monoidal structure order monoids posses and the implicational quantum B-algebras have, as we mentioned in Remark \ref{Discussion2}.
\end{re}
\section{Ending Remarks}
In this last section we briefly sketch a connection between Isbell adjunction, as presented in Section \ref{Intermezzo}, and the construction of injective hulls in the category of topological spaces, as given in \cite{Ban73}. As noticed by \cite{Bar70}, topological spaces can also be seen as examples of ``generalized multicategories'' for the ultrafilter monad. The key difference between multicategories and topological spaces is that, due to the fact that the ultrafilter functor $U$ satisfies $U(1) = 1$, in the latter case it is possible to define the analogue of the covariant presheaf category (see \cite{Hof14}). This allows us to define the analogue of the Isbell adjunction which, unfortunately, does not exist for every topological space. The main result of this brief section is that, given a topological space $X$, if the Isbell adjunction exists then $X$ admits an injective hull. The deep reason why topological spaces and multicategories behave differently, although they are both examples of generalized multicategories, is an interesting open question the author wants to investigate; the hope is to provide a more general theory of injective hulls for $(T,V)$-categories.\par\medskip

Injective topological spaces are characterized in \cite{Sco72} as precisely the continuous
lattices, and in \cite{HOFMANN2011283} it is observed that injective topological spaces are
those spaces where the (topological analogue of the) Yoneda embedding has a left
adjoint in the ordered category of topological spaces and continuous maps. Here
one considers the space
\begin{displaymath}
  \mathbf{2}^{(UX)^{\mathrm{op}}}
\end{displaymath}
where \(\mathbf{2}\) is the Sierpiński space and \(UX\) denotes the set of all
ultrafilters on \(X\) which, with a certain topology, becomes the space
\((UX)^{\mathrm{op}}\). The Yoneda embedding \(\Yo{X} \colon X\to \mathbf{2}^{(UX)^{\mathrm{op}}}\)
sends a point \(x\in X\) to the set \(\{\fx\in UX\mid \fx\to x\}\) of all
ultrafilters convergent to \(x\). In \cite[Example~4.10]{HT10} it is observed
that
\begin{displaymath}
  \mathbf{2}^{(UX)^{\mathrm{op}}} \longrightarrow FX,\quad
  \mathcal{A} \longmapsto\bigcap\mathcal{A}
\end{displaymath}
is an isomorphism; here \(FX=\{\text{filters of open subsets of }X\}\) with the
sets
\begin{displaymath}
  U^{\#}=\{\ff\in FX\mid U\in\ff\},\qquad(U\text{ open in  }X)
\end{displaymath}
forming a basis for the topology of \(FX\) (for instance, see \cite{Esc97}). On
the other hand, the topological analogue to the covariant presheaf category
(see \cite{Hof14}) is the lower Vietoris space
\begin{displaymath}
  VX=\{A\subseteq X\mid A\text{ closed}\};
\end{displaymath}
here the topology is generated by the sets
\begin{displaymath}
  B^{\Diamond}=\{A\in VX\mid A\cap B\neq\varnothing\},\qquad (B\text{ open in }X).
\end{displaymath}

While injective hulls for topological spaces are described in \cite{Ban73}, we point
out here that, similarly to the situation for \(V\)-categories, this description
is ultimately linked to the \emph{Isbell adjunction}. For a topological space
\(X\), the Isbell adjunction is given by the monotone maps
\begin{displaymath}
  \begin{tikzcd}[column sep=huge]
    FX %
    \ar[shift left, start anchor=east, end anchor=west, bend left=25,
    ""{name=U,below}]%
    {r}{(-)^{+}} %
    & VX %
    \ar[start anchor=west,end anchor=east,shift left,bend left=25,
    ""{name=D,above}] %
    {l}{(-)^{-}} \ar[from=U,to=D,"\bot" description]
  \end{tikzcd}
\end{displaymath}
where
\begin{displaymath}
  A^{-}=\bigcap\{\fx\in UX\mid \fx\to x\in A\}
  \quad\text{and}\quad
  \ff^{+}=\lim\ff,
\end{displaymath}
for all closed subsets \(A \subseteq X\) and all filters of open subsets
\(\ff\). Note that \((-)^{-} \colon VX\to FX\) is continuous but
\((-)^{+}\colon FX\to VX\) is in general not. We also recall from \cite{Hof13a}
the following definition.
\begin{defin}
  A topological space \(X\) is \(F\)-core-compact whenever, for all \(x\in X\)
  and all open neighbourhoods \(U\) of \(x\), there exists an open neighbourhood
  \(V\) of \(x\) so that \(V\ll_{F} U\). Here \(V\ll_{F} U\) whenever, for all
  filters \(\ff\) of opens with \(V\in\ff\), \(\lim\ff\cap U\neq\varnothing\).
\end{defin}

\begin{re}
  For open subsets \(U,V\) of \(X\), \(V\ll_{F} U\) if and only if there exists
  some \(x\in X\) with \(V\subseteq \downc x\subseteq U\). We also note that
  \(F\)-core-compact spaces were introduced in \cite{Ern91} under the name
  \emph{C-spaces}. Moreover, in \cite[Proposition~3]{Ban73} it is shown that every \(C\)-space
	admits an injective hull. In fact, \cite[Proposition~3]{Ban73} claims that
	these are exactly the topological spaces which admit an injective hull;
	however, \cite{HM82} (see also \cite{Hof85}) presents a counter example to this
	claim.

\end{re}

\begin{prop}
  A topological space \(X\) is \(F\)-core-compact if and only if the map
  \((-)^{+}\colon FX\to VX\) is continuous.
\end{prop}
\begin{proof}
  Assume first that \(X\) is \(F\)-core-compact. Let \(B\subseteq X\) be open
  and \(\ff\in FX\) with \(\lim\ff\in B^{\Diamond}\), that is
  \(\lim\ff\cap B\neq\varnothing\). Let \(x\in \lim\ff\cap B\). By hypothesis,
  there is an open neighbourhood \(U\) of \(x\) with \(U\ll_{F}V\). Then
  \(\ff\in U^{\#}\) and, for every \(\fg\in U^{\#}\),
  \(\lim\fg\in B^{\Diamond}\).

  Assume now that \((-)^{+}\colon FX\to VX\) is continuous. Let \(x\in X\) and
  let \(B\) be an open neighbourhood of \(x\). Then, with \(\ff\) being the open
  neighbourhood filter of \(x\), \(\ff\to x\in B\), hence
  \(\lim\ff\in B^{\Diamond}\). Since \((-)^{+}\) is continuous, there is some
  open neighbourhood \(U\) of \(x\) so that, for all \(\fg\in U^{\#}\),
  \(\lim\fg\in B^{\Diamond}\).
\end{proof}

With this notation, we can reformulate \cite[Proposition~3]{Ban73}.

\begin{teorema}
	A topological space \(X\) has an injective hull provided that \((-)^{+}\colon
	  FX\to VX\) is continuous.
\end{teorema}

In other words, \(X\) has an injective hull if the Isbell adjunction exists in
the category of topological spaces and continuous maps. In this case,
the injective hull of a space \(X\) is given by the embedding of \(X\) into the
subspace \(\lambda X\) of \(FX\) given by all joins of neighbourhood filters
(see \cite[Proposition~2]{Ban73}). Similarly to the Dedekind-MacNeille completion, these filters are
precisely the fixed points of the Isbell adjunction. To see this, consider
\(\mathbf{2}^{(UX)^{\mathrm{op}}}\simeq FX\) and, for \(\psi\in \mathbf{2}^{(UX)^{\mathrm{op}}}\),
note that
\begin{displaymath}
  (\psi^{-})^{+}=\bigcap\{y_{X}(x)\mid x\in\psi^{-}\}.
\end{displaymath}
Therefore, for \(\ff\in FX\), \(\ff=(\ff^{-})^{+}\) precisely when \(\ff\) is a
join of neighbourhood filters.\\


\section*{Acknowledgements}

I am grateful to D. Hofmann for valuable discussions about the content of the paper and to I. Stubbe for the valuable suggestions he gave me during his visit to Aveiro. The author would like to thank the anonymous referee who kindly reviewed the earlier version of this manuscript and provided valuable suggestions and comments which lead to an overall improvement of the paper. \\
The author acknowledges partial financial assistance by the ERDF – European Regional Development Fund through the Operational Programme for Competitiveness and Internationalisation - COMPETE 2020 Programme and by National Funds through the Portuguese funding agency, FCT - Fundação para a Ciência e a Tecnologia, within project POCI-01-0145-FEDER-030947, and project UID/MAT/04106/2019 (CIDMA). The author is also supported by FCT grant PD/BD/128187/2016.

\end{document}